\documentclass[10pt]{amsart}
\usepackage{amsmath}
\usepackage{amsfonts}
\usepackage{amsmath, amssymb}
\usepackage{amssymb}
\usepackage{mathabx}
\usepackage{wasysym}
\usepackage{latexsym, amscd, amsbsy, xypic, color,verbatim}
\usepackage{mathrsfs,dsfont}
\usepackage{amssymb, bbm}
\usepackage{graphicx, graphics, pstricks}
\usepackage{caption}
\usepackage{pifont}
\usepackage{arydshln}
\usepackage{booktabs}
\usepackage{arydshln}
\usepackage{multirow}
\usepackage{young} 
\usepackage{wasysym}
\usepackage{stmaryrd}
\usepackage{ytableau}
\usepackage{float}
\usepackage{latexsym, amscd, amsbsy, xypic, color,verbatim}
\usepackage{bbm,dsfont}
\usepackage{graphicx, graphics, pstricks}
\usepackage{caption}
\usepackage[mathscr]{eucal}
\usepackage[all]{xy}
\usepackage{fancyhdr}
\usepackage{multicol}
\usepackage{rotating}
\usepackage{lscape}
\usepackage{amssymb,amsthm}
\usepackage{enumitem}
\usepackage[all]{xy}
\usepackage{tikz}
\usepackage{young}
\usepackage[vcentermath,enableskew]{youngtab}
\usepackage{setspace}
\usepackage{subfigure}
\usepackage[mathscr]{eucal}
\usepackage{tkz-euclide}
\usepackage[all]{xy}
\usepackage{fancyhdr}
\usepackage{tikz}
\usepackage{fullpage}
\usepackage{setspace}
\usepackage{footmisc}
\usepackage{subfigure}
\usepackage{youngtab}
\usepackage[colorlinks=true, linkcolor=blue, citecolor=blue, urlcolor=blue]{hyperref}
\usetikzlibrary{decorations.pathreplacing}
\def\thefootnote{\fnsymbol{footnote}}
\newtheorem{thm}{Theorem}[section]
{}
\newtheorem{prop}[thm]{Proposition}
\newtheorem{lemma}[thm]{Lemma}

\newtheorem{defin}[thm]{Definition}
\newtheorem{remark}{remark}[section]

\newcommand{\Hom}{\operatorname{Hom}}

\renewcommand{\le}{\leqslant}
\renewcommand{\ge}{\geqslant}

\newcommand{\Sp}{\textbf{\rm \textbf{S}}}

\newcommand{\arcd}{\ar@{-}@/_/} 
\newcommand{\arcu}{\ar@{-}@/^/} 
\newcommand{\arcU}{\ar@{-}@/^10pt/} 
\newcommand{\tra}{\ar@{-}} 

\newcommand{\xysmall}{\xymatrixrowsep{5pt}\xymatrixcolsep{10pt}\xymatrix} 

 \def\de{\delta}

  \def\leq{\leqslant}  \def\geq{\geqslant}

\def\Hom{\mbox{\rm Hom}}

\def\dim{\mbox{\rm dim}\,}

\def\ggp#1#2{\left[\kern-3.2pt\left[{#1\atop #2}\right]\kern-3.2pt\right]}

\def\le{\leq}
\def\de{\delta}

\def\o{\otimes}

\begin{document}

\title{\bf  Half-Diagrams in Partition Algebras: A Geometric Perspective on Multiplicities}

\date{}

\author{Pei Wang}

\address{Wang: Teachers’ College, Beijing Union University, Beijing 100101, P. R. China}

\email{wangpei19@163.com}

\author{Changjing Zhuge}

\address{Zhuge: Beijing Institute for Scientiﬁc and Engineering Computing,
	Department of Mathematics, School of Mathematics, Statistics and Mechanics, Beijing University of Technology, Beijing, 100124, P. R. China}

\email{zhuge@bjut.edu.cn}

 \maketitle

\begin{abstract}
	This paper studies the restriction multiplicities of half-diagram modules for the partition algebra and their geometric interpretations. By specializing the Bowman–De Visscher–Orellana formula [\cite{BVC}, Theorem 4.3] for restriction multiplicities of standard modules in the partition algebra, we compute these multiplicities and provide interpretations in terms of planar triangles and conic sections. Additionally, through the decomposition of half-diagrams, we explain the intrinsic reasons underlying this connection between geometry and algebra.

\end{abstract}

\renewcommand{\thefootnote}{\alph{footnote}}
\setcounter{footnote}{-1} 
\renewcommand{\thefootnote}{\alph{footnote}}
\setcounter{footnote}{-1} 
\setcounter{footnote}{-1}

\section{Introduction}
In representation theory, the concept of multiplicity classically refers to the number of times an irreducible module appears as a composition factor in a given module. This problem is deeply rooted in the Jordan–Hölder theorem, which asserts that  the composition factors of a module are uniquely determined. For the symmetric group $\mathfrak{S}_n$, the irreducible Specht  modules   $\Sp(\mu)$ over the complex numbers  $\mathbb{C}$, indexed by integer partitions $\mu \vdash n$, provide a natural framework for studying multiplicities. The Littlewood–Richardson rule provides a combinatorial formula for the multiplicities $c_{\lambda,\mu}^\nu$, counting how often the outer tensor product $\Sp(\lambda) \boxtimes \Sp(\mu)$ ( $\lambda \vdash m$,  $\mu \vdash n$) appears in the restriction of $\Sp(\nu)$ ( $\nu \vdash (m+n)$)  to a Young subgroup $\mathfrak{S}_m \times \mathfrak{S}_n$. The reader is referred to Ref. \cite{jk}. 

The diagram algebras --- the Brauer algebra, the partition algebra, and the Temperley--Lieb algebra --- extend this framework to standard  (or cell)  modules (see definition in Section 2.3). König and Xi’s inflation theory \cite{KX0,KX1} constructs these algebras as iterated inflations of group algebras of the symmetric groups using modules spanned by \textit{half-diagrams} (see \S2.2). Consequently, their restriction multiplicities often reduce to combinations of Specht module multiplicities. For example, Hartmann and Paget \cite{HP1} studied restriction and induction of standard modules for Brauer algebras, proving that multiplicities in their standard module filtrations are well-defined. Bowman, De Visscher, and Orellana \cite{BVC}  derived an explicit formula (see \S3.1) for partition algebras involving three-part Littlewood--Richardson coefficients and Kronecker coefficients.

In the work \cite{WP}, the multiplicities of standard modules for the Temperley--Lieb algebra $\mathrm{TL}_n(\delta)$ were studied. A distinctive feature of this algebra is that its standard modules are precisely spanned by half-diagrams. Let \(V_n(r)\) denote the half-diagram module of \(\mathrm{TL}_n(\delta)\) indexed by a non-negative integer \(r\), where \(r\) represents the number of labeled dots (see Section 6, Fig.\ref{2PD}).
Remarkably, in the semisimple case, when \(V_{m+n}(r)\) is viewed as a \(\mathrm{TL}_m \otimes \mathrm{TL}_n\)-module, the multiplicity $E_{p,q}^r$ of \(V_m(p) \boxtimes V_n(q)\) in \(V_{m+n}(r)\) is characterized by planar triangles: \(E_{p,q}^r  = 1\) occurs precisely when \(p, q, r\) can form a triangle (including degenerate cases where equality holds in the triangle inequality or side lengths vanish); otherwise \(E_{p,q}^r= 0\).
 This equivalence transforms a representation-theoretic problem into a geometric one.

A natural question arises:  
\textbf{For the partition algebra, do the restriction multiplicities of its half-diagram modules (which correspond to standard modules indexed by one-part partitions) admit a similar geometric interpretation?}

We answer the question in the affirmative. Specializing the Bowman–De Visscher–Orellana formula [\cite{BVC}, Theorem 4.3] to one-part partitions reduces the original multiplicity problem to a  Diophantine system, denoted (E1) in \S 3.3.  The multiplicity $E_{p,q}^r$ is thereby identified with the number of non-negative integer solutions of (E1).  Solving this system not only settles the arithmetic question but also uncovers surprising links with planar triangles and conic sections.

Furthermore, adopting the diagrammatic approach of Bowman, De Visscher, and Orellana, we investigate the restriction of `walled' half-diagrams. The resulting decomposition provides an intuitive diagrammatic interpretation of the variables in the Diophantine system (E1). This interpretation aims to clarify the combinatorial meaning of these unknowns and to explain why each non-negative integer solution to (E1) contributes exactly 1 to the multiplicity \(E_{p, q}^{r}\) when its values are assigned to the corresponding indices.

The remainder of the paper is organized as follows.  Section 2 recalls partition algebras together with their standard modules.  Section 3 derives the Diophantine system (E1) from the specialized Bowman–De Visscher–Orellana formula.  Section 4 interprets (E1) through the decomposition of walled half-diagrams, while Section 5 translates these combinatorial data into geometric language.  Finally, Section 6 extends the discussion to Temperley–Lieb algebras and their Grothendieck groups.

\section{Partition algebra}
The partition algebra, introduced independently around 1990 by Martin \cite{Martin94} and Jones \cite{Jones1} to model the Potts model in statistical mechanics, finds widespread application across mathematics and physics, with related developments in \cite{Martin91,Martin96,Martin00}. It has been extensively studied from multiple perspectives: fundamental structural and algebraic properties, including generators-and-relations presentations and monoid structure \cite{HR05,East}, the basic construction and semisimplicity criteria \cite{HR05}, quasi-hereditary structure \cite{Martin96}, cellularity \cite{Xi}, and block classification in positive characteristic \cite{BDVK15}; representation-theoretic and combinatorial aspects, including 
Robinson–Schensted correspondence for the partition algebra \cite{MR98},
Murphy elements and Specht modules \cite{HR05}, character theory and dimensions of irreducible modules \cite{H,FaH,BHH}, and combinatorial models such as RSK insertion and set-partition tableaux for diagram-algebra representations \cite{HL05,HJ20}; and developments in invariant theory via Schur–Weyl duality with the symmetric group \cite{HR05}, including generalizations of the fundamental theorems of invariant theory in the partition-algebra setting \cite{BH,BDM22II} and the establishment of an integral Schur–Weyl duality \cite{BDM22}. Connections to classical symmetric-group representation theory have also been explored, notably through the study of Kronecker coefficients using partition algebras \cite{BVC}. Moreover, its subalgebras, including the Temperley–Lieb algebra \cite{TL} and the Brauer algebra \cite{Br}, play significant roles in areas such as invariant theory and knot theory.

\subsection{Definitions}

Let $\mathbb{N}^+$ denote the set of positive integers, and fix $n\in\mathbb{N}^+$. Define $M=\{1,\ldots,n,,1',\ldots,n'\}$.
A set-partition $d$ of  $M$ is a collection of subsets of $M$ such that

(1) The intersection of any two different subsets in the collection of subsets is an empty set, and, 

(2) The union of all sets in the collection of subsets is equal to $M$.

Each subset of $M$ in a partition $d$  is called a \textit{block}. 

A set-partition $d$ of $M$
can  be represented by a diagram called an $(n, n)$-\textit{partition diagram} (or a partition diagram of degree $n$).
A partition diagram consists of a rectangular frame with marked dots $1,2, \ldots, n$, placed on the upper boundary and marked dots $1^{\prime}, 2^{\prime}, \ldots, n^{\prime}$ placed on the lower boundary (reading both from left to right) and a collection of strings pairing up various dots.
Such a diagram represents $d$ if its connected components correspond to the blocks of $d$, and two such diagrams are considered equivalent if their connected components agree.
Figure \ref{PD} illustrates two equivalent diagrams (parts (1) and (2)) that represent the set partition  $d=\left\{\left\{1,2^{\prime}, 3^{\prime}\right\},\{2,3,4\},\left\{1^{\prime}\right\},\left\{4^{\prime}\right\}\right\}$.

\begin{figure}[ht]
	\begin{equation*}
		\begin{array}{c}
			\begin{tikzpicture}
			\draw [green,thick](-0.4,1.5) --(2.8,1.5);
			\draw [green,thick](-0.4,0) --(2.8,0);
			\draw [green,thick](-0.4,1.5) --(-0.4,0);
			\draw [green,thick](2.8,1.5) --(2.8,0);
			\node at (0,1.8) {$1$}; 
			\node at (0.8,1.8) {$2$};
			\node at (1.6,1.8) {$3$};
			\node at (2.4,1.8) {$4$};
			\node at (0,-0.3) {$1'$}; 
			\node at (0.8,-0.3) {$2'$};
			\node at (1.6,-0.3) {$3'$};
			\node at (2.4,-0.3) {$4'$};
			\node at (0,1.5) {\textbullet}; 
			\node at (0.8,1.5) {\textbullet};
			\node at (1.6,1.5) {\textbullet};
			\node at (2.4,1.5) {\textbullet};
			\node at (0,0) {\textbullet}; 
			\node at (0.8,0) {\textbullet};
			\node at (1.6,0) {\textbullet};
			\node at (2.4,0) {\textbullet};
			\draw[thick] (0.8,1.5) .. controls (1,1.1) and (1.4,1.1) ..  (1.6,1.5); 
			\draw [thick](0,1.5) .. controls (0.2,0.8) and (0.6,0.4) ..  (0.8,0); 
			\draw [thick](0.8,0) .. controls (1,0.4) and (1.4,0.4) ..  (1.6,0); 
			\draw [thick](1.6,1.5) .. controls (1.8,1.1) and (2.2,1.1) ..  (2.4,1.5);
			\node at (1.2,-0.7) { (1)};
			\node at (3,-1.2) { Two equivalent  diagrams};
			\node at (4.8,-0.7) { (2)};
			\draw [green,thick](3.2,1.5) --(6.4,1.5); 
			\draw [green,thick](3.2,0) --(6.4,0);
			\draw [green,thick](3.2,1.5) --(3.2,0);
			\draw [green,thick](6.4,1.5) --(6.4,0);
			\node at (3.6,1.8) {$1$}; 
			\node at (4.4,1.8) {$2$};
			\node at (5.2,1.8) {$3$};
			\node at (6,1.8) {$4$};
			\node at (3.6,-0.3) {$1'$}; 
			\node at (4.4,-0.3) {$2'$};
			\node at (5.2,-0.3) {$3'$};
			\node at (6,-0.3) {$4'$};
			\node at (3.6,1.5) {\textbullet}; 
			\node at (4.4,1.5) {\textbullet};
			\node at (5.2,1.5) {\textbullet};
			\node at (6,1.5) {\textbullet};
			\node at (3.6,0) {\textbullet}; 
			\node at (4.4,0) {\textbullet};
			\node at (5.2,0) {\textbullet};
			\node at (6,0) {\textbullet};
			\draw [thick](4.4,1.5) .. controls (4.6,0.7) and (5.8,0.7) ..  (6,1.5); 
			\draw [thick](3.6,1.5) .. controls (3.8,0.8) and (4.2,0.4) ..  (4.4,0); 
			\draw [thick](3.6,1.5) .. controls (3.8,0.8) and (4.8,0.8) ..  (5.2,0); 
			\draw [thick](4.4,0) .. controls (4.6,0.4) and (5,0.4) ..  (5.2,0); 
			\draw [thick](5.2,1.5) .. controls (5.4,1.1) and (5.8,1.1) ..  (6,1.5);
		\end{tikzpicture}
		\end{array}
		\begin{array}{c}
				\begin{tikzpicture}
				\draw [green,thick](-0.4,1.5) --(2.8,1.5);
				\draw [green,thick](-0.4,0) --(2.8,0);
				\draw [green,thick](-0.4,1.5) --(-0.4,0);
				\draw [green,thick](2.8,1.5) --(2.8,0);
				\node at (0,1.8) {$1$}; 
				\node at (0.8,1.8) {$2$};
				\node at (1.6,1.8) {$3$};
				\node at (2.4,1.8) {$4$};
				\node at (0,-0.3) {$1'$}; 
				\node at (0.8,-0.3) {$2'$};
				\node at (1.6,-0.3) {$3'$};
				\node at (2.4,-0.3) {$4'$};
				\node at (0,1.5) {\textbullet}; 
				\node at (0.8,1.5) {\textbullet};
				\node at (1.6,1.5) {\textbullet};
				\node at (2.4,1.5) {\textbullet};
				\node at (0,0) {\textbullet}; 
				\node at (0.8,0) {\textbullet};
				\node at (1.6,0) {\textbullet};
				\node at (2.4,0) {\textbullet};
				\draw [thick](0,1.5) .. controls (0.2,1.1) and (0.6,1.1) ..  (0.8,1.5); 
				\draw [thick](0.8,1.5) .. controls (1,0.8) and (1.2,0.4) ..  (1.6,0); 
				\node at (1.2,-0.7) { (3)};
				\draw [thick](0,0) .. controls (0.2,0.4) and (0.6,0.4) ..  (0.8,0); 
				\draw [thick](1.6,1.5) .. controls (1.8,1.1) and (2.2,1.1) ..  (2.4,1.5);
				\draw [thick](0.8,0) .. controls (0.8,1.4) and (2.5,0.6) ..  (2.4,1.5); 
				\draw [red,dashed,thick](1.1,0.7) circle (8pt);
				\node [red]at (2.15,0.7) {\footnotesize\mbox{-crossing}};
				\node at (1.2,-1.2) { Crossing diagram};
			\end{tikzpicture}
	\end{array}
	\end{equation*}
	\vskip -0.3cm
	\caption{Partition diagrams }
	\label{PD}
\end{figure}
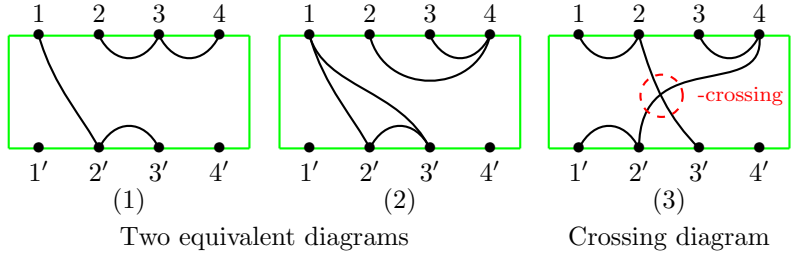

          We further impose the following order on $M$: $1<2<\cdots<n<n'<\cdots<1'$. 
          Recall that two blocks $B_1$ and $B_2$ are crossing if there exist $k_1$ and $k_2$ in block $B_1$, and $k_3$ and $k_4$ in $B_2$, which satisfy $k_1<k_3<k_2<k_4$. As shown in Figure 1(3), blocks $\{1,2,3'\}$ and $\{3,4,1',2'\}$ are crossing. 
          
          If every pair of blocks in a set-partition $d$ is non-crossing, then $d$ is called non-crossing (e.g., see Figure \ref{PD}(1) and (2) for illustrations). Crossings can typically be described using the symmetric group, which will be introduced later. However, the focus of Section 6, the Temperley-Lieb diagrams, involves non-crossing configurations.

Let $\mathbb{C}$ be the field of complex numbers, and let $\delta \in \mathbb{C}$ be a fixed number.
  The partition algebra $P_n(\delta)$ of degree $n$,  which we simply write as $P_n$, is first a  $\mathbb{C}$-vector space spanned by all  set-partitions of $M$ as a basis.  Multiplication is then defined by the `concatenation' of diagrams. Let $d_1$ and $d_2$ be two  partition diagrams. Then $d_1\cdotp d_2$ is defined as follows: first, place $d_1$ above $d_2$ so that the dots on $d_1$’s lower boundary coincide with those on $d_2$’s upper boundary (see Figure 2). The resulting diagram may include interior components (i.e., components that are not connected to the upper boundary of $d_1$ or the lower boundary of $d_2$). Suppose there are $t$ interior components, and let $d'$ be the diagram obtained by removing all such interior components. Then define the product as $d_1 \cdot d_2 = \delta^t d'$.

Figure \ref{MD} illustrates the product $d_1 \cdot d_2$. After stacking the two diagrams, two interior components appear (highlighted in red).

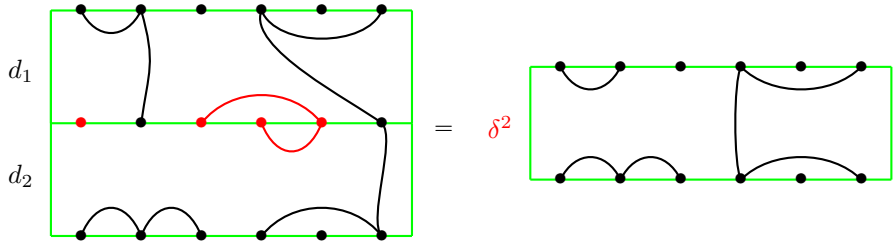
\begin{figure}[h]
	\begin{equation*}
		\begin{array}{c}
			\begin{tikzpicture}
				\draw [green,thick](-0.4,1.5) --(4.4,1.5);
				\draw [green,thick](-0.4,0) --(4.4,0);
				\draw [green,thick](-0.4,1.5) --(-0.4,0);
				\draw [green,thick](4.4,1.5) --(4.4,0);
					\draw [green,thick](-0.4,1.5) --(-0.4,-1.5);
						\draw [green,thick](-0.4,-1.5) --(4.4,-1.5);
							\draw [green,thick](4.4,-1.5) --(4.4,0);
		
				\node at (0,1.5) {\textbullet}; 
				\node at (0.8,1.5) {\textbullet};
				\node at (1.6,1.5) {\textbullet};
				\node at (2.4,1.5) {\textbullet};
					\node at (3.2,1.5) {\textbullet};
						\node at (4,1.5) {\textbullet};
				\node [red] at (0,0) {\textbullet}; 
				\node at (0.8,0) {\textbullet};
				\node [red]at (1.6,0) {\textbullet};
				\node [red]at (2.4,0) {\textbullet};
				\node [red]at (3.2,0) {\textbullet};
				\node at (4,0) {\textbullet};
				\node at (0,-1.5) {\textbullet}; 
				\node at (0.8,-1.5) {\textbullet};
				\node at (1.6,-1.5) {\textbullet};
				\node at (2.4,-1.5) {\textbullet};
				\node at (3.2,-1.5) {\textbullet};
				\node at (4,-1.5) {\textbullet};
				\draw [thick](0,1.5) .. controls (0.2,1.1) and (0.6,1.1) ..  (0.8,1.5); 
				\draw [thick](0.8,1.5) .. controls (1,0.8) and (0.9,0.7) ..  (0.8,0); 
				\draw [thick](2.4,1.5) .. controls (2.6,1) and (3.8,1) ..  (4,1.5); 
					\draw [thick](2.4,1.5) .. controls (2.2,1) and (3.6,0.3) ..  (4,0); 
						\draw [red,thick](1.6,0) .. controls (2,0.5) and (2.8,0.5) ..  (3.2,0); 
						\draw [red,thick](2.4,0) .. controls (2.6,-0.5) and (3,-0.5) ..  (3.2,0); 
							\draw [thick](0,-1.5) .. controls (0.2,-1) and (0.6,-1) ..  (0.8,-1.5); 
								\draw [thick](0.8,-1.5) .. controls (1,-1) and (1.4,-1) ..  (1.6,-1.5); 
									\draw [thick](2.4,-1.5) .. controls (2.8,-1) and (3.6,-1) ..  (4,-1.5); 
										\draw [thick](4,0) .. controls (4.2,-0.2) and (3.8,-1.3) ..  (4,-1.5); 
					\node at (-0.8,0.7) {$d_1$}; 
						\node at (-0.8,-0.7) {$d_2$}; 
			\end{tikzpicture}
		\end{array} =\, 
		\begin{array}{c}
			\begin{tikzpicture}
				\draw [green,thick](-0.4,1.5) --(4.4,1.5);
			\draw [green,thick](-0.4,0) --(4.4,0);
			\draw [green,thick](-0.4,1.5) --(-0.4,0);
			\draw [green,thick](4.4,1.5) --(4.4,0);

			\node at (0,1.5) {\textbullet}; 
			\node at (0.8,1.5) {\textbullet};
			\node at (1.6,1.5) {\textbullet};
			\node at (2.4,1.5) {\textbullet};
			\node at (3.2,1.5) {\textbullet};
			\node at (4,1.5) {\textbullet};
			\node at (0,0) {\textbullet}; 
			\node at (0.8,0) {\textbullet};
			\node at (1.6,0) {\textbullet};
			\node at (2.4,0) {\textbullet};
			\node at (3.2,0) {\textbullet};
			\node at (4,0) {\textbullet};
			\draw [thick](0,1.5) .. controls (0.2,1.1) and (0.6,1.1) ..  (0.8,1.5); 
			\draw [thick](0,0) .. controls (0.2,0.4) and (0.6,0.4) ..  (0.8,0); 
					\draw [thick](0.8,0) .. controls (1,0.4) and (1.4,0.4) ..  (1.6,0); 
	\draw [thick](2.4,0) .. controls (2.8,0.4) and (3.6,0.4) ..  (4,0); 
		\draw [thick](2.4,1.5) .. controls (2.8,1.1) and (3.6,1.1) ..  (4,1.5); 
			\draw [thick](2.4,1.5) .. controls (2.3,1.3) and (2.3,0.2) ..  (2.4,0); 
					\node [red] at (-0.8,0.7) {$\delta^2$}; 
			\end{tikzpicture}
		\end{array}
	\end{equation*}
	\vskip -0.3cm
	\caption{Multiplication of diagrams }
	\label{MD}
\end{figure}

It is known that $P_n$ is generated by the elements $e_{i,j}$, $p_i$, $s_{i,j}$$(1\leq i<n,1\leq j\leq n)$ and illustrated in Figure \ref{Gene} below.

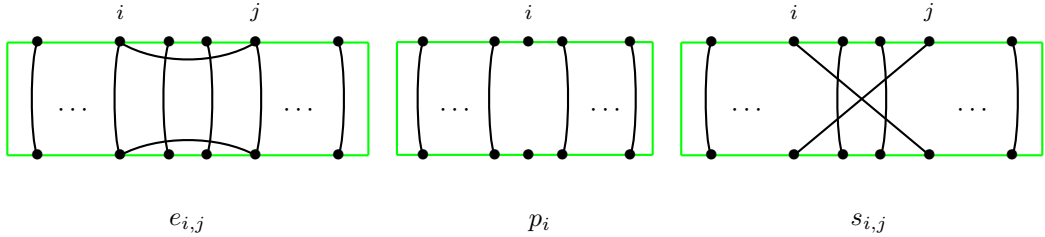
\begin{figure}[h]
	\begin{equation*}
		\begin{array}{c}
			\begin{tikzpicture}
			\draw [green,thick](-0.4,1.5) --(4.4,1.5);
			\draw [green,thick](-0.4,0) --(4.4,0);
			\draw [green,thick](-0.4,1.5) --(-0.4,0);
			\draw [green,thick](4.4,1.5) --(4.4,0);

			\node at (0,1.5) {\textbullet}; 
			\node at (1.1,1.5) {\textbullet};
			\node at (1.75,1.5) {\textbullet};
			\node at (2.25,1.5) {\textbullet};
					\node at (1.1,1.9) { \footnotesize$i$};
						\node at (2.9,1.9) { \footnotesize$j$};
			\node at (2.9,1.5) {\textbullet};
			\node at (4,1.5) {\textbullet};
			\node at (0,0) {\textbullet}; 
			\node at (1.1,0) {\textbullet};
			\node at (1.75,0) {\textbullet};
			\node at (2.25,0) {\textbullet};
			\node at (2.9,0) {\textbullet};
			\node at (4,0) {\textbullet};
				\node at (0.5,0.6) { $\cdots$};
					\node at (3.5,0.6) { $\cdots$};
		\draw [thick](0,1.5) .. controls (-0.1,1.3) and (-0.1,0.2) ..  (0,0); 
			\draw [thick](1.1,1.5) .. controls (1,1.3) and (1,0.2) ..  (1.1,0); 
			\draw [thick](2.25,1.5) .. controls (2.35,1.3) and (2.35,0.2) ..  (2.25,0); 
				\draw[thick] (2.9,1.5) .. controls (3,1.3) and (3,0.2) ..  (2.9,0); 
					\draw [thick](4,1.5) .. controls (4.1,1.3) and (4.1,0.2) ..  (4,0); 
				\draw [thick](1.75,1.5) .. controls (1.65,1.3) and (1.65,0.2) ..  (1.75,0); 
				\node at (2,-0.9) { $e_{i,j}$};
				\draw [thick](1.1,1.5) .. controls (1.6,1.2) and (2.4,1.2) ..  (2.9,1.5);
					\draw [thick](1.1,0) .. controls (1.6,0.27) and (2.4,0.27) ..  (2.9,0);
				
			\end{tikzpicture}
		\end{array}
		\begin{array}{c}
			\begin{tikzpicture}
				\draw [green,thick](-0.2,1.5) --(3.2,1.5);
			\draw [green,thick](-0.2,0) --(3.2,0);
			\draw [green,thick](-0.2,1.5) --(-0.2,0);
			\draw [green,thick](3.2,1.5) --(3.2,0);

			\node at (0.15,1.5) {\textbullet}; 
			\node at (1.1,1.5) {\textbullet};
			\node at (1.55,1.5) {\textbullet};
			\node at (1.55,1.9) { \footnotesize$i$};
						\node at (2,1.5) {\textbullet};
									\node at (2,0) {\textbullet};
			\node at (2.9,1.5) {\textbullet};
			\node at (0.15,0) {\textbullet}; 
			\node at (1.1,0) {\textbullet};
			\node at (1.55,0) {\textbullet};
			\node at (2.9,0) {\textbullet};
			\node at (0.6,0.6) { $\cdots$};
			\node at (2.6,0.6) { $\cdots$};
			\draw [thick](0.15,1.5) .. controls (0.05,1.3) and (0.05,0.2) ..  (0.15,0); 
			\draw [thick](1.1,1.5) .. controls (1,1.3) and (1,0.2) ..  (1.1,0); 
				\draw [thick](2,1.5) .. controls (2.1,1.3) and (2.1,0.2) ..  (2,0); 
			\draw [thick](2.9,1.5) .. controls (3,1.3) and (3,0.2) ..  (2.9,0); 
				\node at (1.7,-0.9) { $p_i$};
			\end{tikzpicture}
		\end{array}
		\begin{array}{c}
				\begin{tikzpicture}
				\draw [green,thick](-0.4,1.5) --(4.4,1.5);
				\draw [green,thick](-0.4,0) --(4.4,0);
				\draw [green,thick](-0.4,1.5) --(-0.4,0);
				\draw [green,thick](4.4,1.5) --(4.4,0);
	
				\node at (0,1.5) {\textbullet}; 
				\node at (1.1,1.5) {\textbullet};
				\node at (1.75,1.5) {\textbullet};
				\node at (2.25,1.5) {\textbullet};
				\node at (1.1,1.9) { \footnotesize$i$};
				\node at (2.9,1.9) { \footnotesize$j$};
				\node at (2.9,1.5) {\textbullet};
				\node at (4,1.5) {\textbullet};
				\node at (0,0) {\textbullet}; 
				\node at (1.1,0) {\textbullet};
				\node at (1.75,0) {\textbullet};
				\node at (2.25,0) {\textbullet};
				\node at (2.9,0) {\textbullet};
				\node at (4,0) {\textbullet};
				\node at (0.5,0.6) { $\cdots$};
				\node at (3.5,0.6) { $\cdots$};
				\draw [thick](0,1.5) .. controls (-0.1,1.3) and (-0.1,0.2) ..  (0,0); 
				\draw [thick](2.25,1.5) .. controls (2.35,1.3) and (2.35,0.2) ..  (2.25,0); 
				\draw [thick](4,1.5) .. controls (4.1,1.3) and (4.1,0.2) ..  (4,0); 
				\draw [thick](1.75,1.5) .. controls (1.65,1.3) and (1.65,0.2) ..  (1.75,0); 
				\node at (2.1,-0.9) { $s_{i,j}$};
				\draw [thick](1.1,1.5) -- (2.9,0);
				\draw [thick](1.1,0) -- (2.9,1.5);
						\end{tikzpicture}
	\end{array}
	\end{equation*}
	\vskip -0.3cm
	\caption{Generators of $P_n$}
	\label{Gene}
\end{figure}

A block is called a \textit{propagating block} if it contains dots on both the upper and lower boundaries. For example, in Figure \ref{PD}(1), the block $\{1,2',3'\}$ is propagating.
The propagating number of a partition diagram $d$, denoted by $pn(d)$, is the number of propagating blocks in the diagram.
During concatenation, interior components are removed and no new  component can connect  the two external boundaries; existing propagating blocks may merge but 
cannot increase in number, hence the propagating number never increases:
$$
pn(d_1\cdotp d_2)\le \min\{pn(d_1),\,pn(d_2)\}.
$$

\subsection{Half-Diagrams and Inflation}

Given a partition diagram of degree \(n\) with \(r\) propagating blocks, let us attempt to decompose it. Because a partition diagram can be represented in non-unique ways, we first \textit{standardize} the diagram. Specifically, for each propagating block, we use exactly one string to connect its upper part to its lower part, choosing the leftmost dots in the upper and lower sub-blocks as the string’s endpoints.

Subsequently, we cut all propagating strings and replace the dots within the propagating blocks with white dots –  called \textit{labeled dots} – to indicate their previous connections to both upper and lower blocks. Blocks containing such labeled dots are termed \textit{labeled blocks} (see Figure \ref{De} for an example).  
Through this operation, we obtain two distinct half-diagrams (upper and lower), collectively referred to as an \((n, r)\)-\textit{half-diagram}. Formally, an \((n, r)\)-half-diagram is a set partition of \(\{1, 2, \dots, n\}\) featuring exactly \(r\) labeled blocks. The vector space spanned by all such \((n, r)\)-half-diagrams, with this set as basis, is denoted \(V_n(r)\).  

The middle segment of the partition diagram corresponds uniquely to an element of the symmetric group \(\mathfrak{S}_r\), with this correspondence being uniquely determined through diagram standardization.

\begin{figure}[h]
	\begin{equation*}
		\begin{array}{c}
			\begin{tikzpicture}
				\draw [green,thick](-0.4,1.5) --(3.8,1.5);
				\draw [green,thick](-0.4,0) --(3.8,0);
				\draw [green,thick](-0.4,1.5) --(-0.4,0);
				\draw [green,thick](3.8,1.5) --(3.8,0);

					\node at (0,1.8) {1}; 
				\node at (0.5,1.8) {2}; 
				\node at (1,1.8) {3};
				\node at (1.5,1.8) {4};
				\node at (2,1.8) {5};
				\node at (2.5,1.8) {6};
				\node at (3,1.8) {7};
				\node at (3.5,1.8) {8};
					\node at (0,1.5) {\textbullet}; 
				\node at (0.5,1.5) {\textbullet}; 
				\node at (1,1.5) {\textbullet};
				\node at (1.5,1.5) {\textbullet};
				\node at (2,1.5) {\textbullet};
				\node at (2.5,1.5) {\textbullet};
				\node at (3,1.5) {\textbullet};
				\node at (3.5,1.5) {\textbullet};
					\node at (0,-0.3) {1$'$}; 
				\node at (0.5,-0.3) {2$'$}; 
				\node at (1,-0.3) {3$'$};
				\node at (1.5,-0.3) {4$'$};
				\node at (2,-0.3) {5$'$};
				\node at (2.5,-0.3) {6$'$};
				\node at (3,-0.3) {7$'$};
				\node at (3.5,-0.3) {8$'$};
					\node at (0,0) {\textbullet}; 
			\node at (0.5,0) {\textbullet}; 
			\node at (1,0) {\textbullet};
			\node at (1.5,0) {\textbullet};
			\node at (2,0) {\textbullet};
			\node at (2.5,0) {\textbullet};
			\node at (3,0) {\textbullet};
			\node at (3.5,0) {\textbullet};
				\draw [red,thick](0,1.5) .. controls (0.5,1) and (1,1) ..  (1.5,1.5); 
             	\draw  [red,thick](0,1.5) .. controls (0,0.8) and (0.5,0.6) ..  (0.5,0); 
             	    	\draw [thick](1,1.5) .. controls (1.2,1.1) and (1.8,1.1) ..  (2,1.5); 
             			\draw [thick](2,1.5) .. controls (2.2,1.28) and (2.3,1.28) ..  (2.5,1.5); 
             			\draw  [thick](1,1.5) .. controls (1,1.1) and (0,0.8) ..  (0,0); 
             			      	\draw   [blue,thick](3.5,1.5) .. controls (3.5,0.7) and (1,1) ..  (1,0); 
             			      	\draw [blue,thick] (1,0).. controls (1.1,0.3) and (1.9,0.3) ..  (2,0); 
             			      	 	\draw [purple,thick] (1.5,0).. controls (1.6,0.5) and (2.9,0.5) ..  (3,0); 
             			     
	\draw [green,dashed,thick](-0.8,1) --(4.2,1);
	\draw [green,dashed,thick](-0.8,0.5) --(4.2,0.5);			
			
			\end{tikzpicture}
		\end{array}
		\begin{array}{c}
			\begin{tikzpicture}
			\draw [green,thick](-0.4,1.5) --(3.8,1.5);
			\draw [green,thick](-0.4,0) --(3.8,0);

			\node at (0,1.8) {1}; 
			\node at (0.5,1.8) {2}; 
			\node at (1,1.8) {3};
			\node at (1.5,1.8) {4};
			\node at (2,1.8) {5};
			\node at (2.5,1.8) {6};
			\node at (3,1.8) {7};
			\node at (3.5,1.8) {8};
			\node at (0,1.46) {{\huge\textopenbullet}}; 
			\node at (0.5,1.5) {\textbullet}; 
			\node at (1,1.46) {{\huge\textopenbullet}};
			\node at (1.5,1.46) {{\huge\textopenbullet}};
			\node at (2,1.46) {{\huge\textopenbullet}};
			\node at (2.5,1.46) {{\huge\textopenbullet}};
			\node at (3,1.5) {\textbullet};
			\node at (3.5,1.5) {\textbullet};
			\node at (0,-0.3) {1$'$}; 
			\node at (0.5,-0.3) {2$'$}; 
			\node at (1,-0.3) {3$'$};
			\node at (1.5,-0.3) {4$'$};
			\node at (2,-0.3) {5$'$};
			\node at (2.5,-0.3) {6$'$};
			\node at (3,-0.3) {7$'$};
			\node at (3.5,-0.3) {8$'$};
			\node at (0,-0.05) {{\huge\textopenbullet}}; 
			\node at (0.5,-0.05) {{\huge\textopenbullet}}; 
			\node at (1,-0.05) {{\huge\textopenbullet}};
			\node at (1.5,0) {\textbullet};
			\node at (2,-0.05) {{\huge\textopenbullet}};
			\node at (2.5,0) {\textbullet};
			\node at (3,0) {\textbullet};
			\node at (3.5,0) {\textbullet};
			\draw [red,thick](0,1.5) .. controls (0.5,1) and (1,1) ..  (1.5,1.5); 
			\draw [thick](1,1.5) .. controls (1.2,1.1) and (1.8,1.1) ..  (2,1.5); 
			\draw [thick](2,1.5) .. controls (2.2,1.28) and (2.3,1.28) ..  (2.5,1.5); 
			\draw [blue,thick] (1,0).. controls (1.1,0.3) and (1.9,0.3) ..  (2,0); 
			\draw [purple,thick] (1.5,0).. controls (1.6,0.5) and (2.9,0.5) ..  (3,0); 
			
				\node at (-1,0.7) { $\longleftrightarrow$};	
		\end{tikzpicture}
		\end{array}
			\begin{array}{c}
		\begin{pmatrix}
			1 & 2& 3 \\
			2 & 1 & 3 \\
		\end{pmatrix}	\in \Sigma_{3}
	\end{array}
	\end{equation*}
	\vskip -0.3cm
	\caption{Diagram decomposition}
	\label{De}
\end{figure}
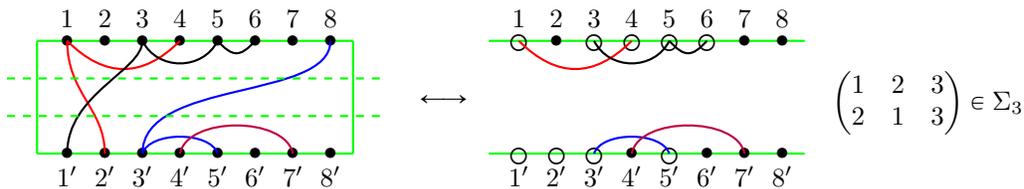

 A partition diagram can act on a half-diagram, thereby making the space of half-diagrams into a $P_n$-module.
 Let $d$ be a partition diagram in $P_n$ and $v$ be an $(n, r)$-half-diagram. The action  $d.v$ of $d$ on $v$ is defined as the following. First, the diagram $d$ stacks above the diagram $v$, and the dots of $d$'s lower boundary coincides with the dots of $v$ to form the new diagram $v'$. The top row, $top(v')$, in the diagram $v'$ is still a half-diagram, and a block in the half-diagram is a labeled block if and only if it is connected to the labeled block in $v$. Furthermore, $v'$ may contain components that are not connected to the top row, and it is assumed that it contains a total of $t$ such components. If $top(v')$ has $r$ labeled blocks, then $d.v=\delta^t top(v')$; otherwise, $d.v=0$. In fact, the number of labeled blocks will not increase during the preceding process. Figure \ref{action} is an illustration. Following the action, there are two labeled blocks and one component (marked in red) that is not connected to the top row.

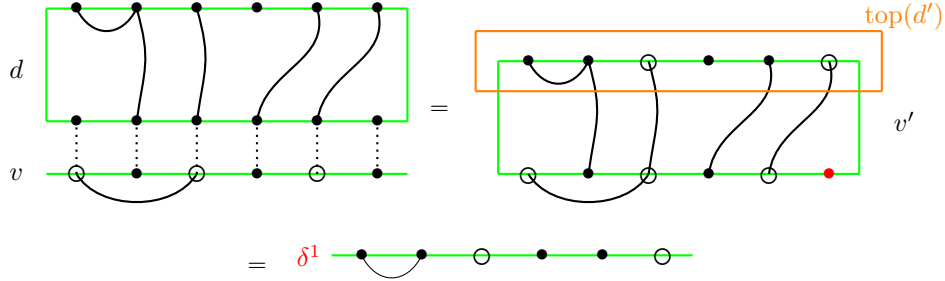
\begin{figure}[ht]
	\begin{equation*}
		\begin{array}{c}
			\begin{tikzpicture}
				\draw [green,thick](-0.4,1.5) --(4.4,1.5);
				\draw [green,thick](-0.4,0) --(4.4,0);
				\draw [green,thick](-0.4,1.5) --(-0.4,0);
				\draw [green,thick](4.4,1.5) --(4.4,0);
			\draw [green,thick](-0.4,-0.7) --(4.4,-0.7);
	\draw [dotted,thick](0,-0.7) --(0,0);
\draw [dotted,thick](0.8,-0.7) --(0.8,0);
\draw [dotted,thick](1.6,-0.7) --(1.6,0);
\draw [dotted,thick](2.4,-0.7) --(2.4,0);
\draw [dotted,thick](3.2,-0.7) --(3.2,0);
\draw [dotted,thick](4,-0.7) --(4,0);
				\node at (0,1.5) {\textbullet}; 
				\node at (0.8,1.5) {\textbullet};
				\node at (1.6,1.5) {\textbullet};
				\node at (2.4,1.5) {\textbullet};
				\node at (3.2,1.5) {\textbullet};
				\node at (4,1.5) {\textbullet};
				\node  at (0,0) {\textbullet}; 
				\node at (0.8,0) {\textbullet};
				\node at (1.6,0) {\textbullet};
				\node at (2.4,0) {\textbullet};
				\node at (3.2,0) {\textbullet};
				\node at (4,0) {\textbullet};
				\node at (0,-0.72) {{\huge\textopenbullet}}; 
				\node at (0.8,-0.7) {\textbullet};
				\node at (1.6,-0.72) {{\huge\textopenbullet}};
				\node at (2.4,-0.7) {\textbullet};
				\node at (3.2,-0.72) {{\huge\textopenbullet}};
				\node at (4,-0.7) {\textbullet};
				\draw [thick](0,1.5) .. controls (0.2,1.1) and (0.6,1.1) ..  (0.8,1.5); 
				\draw [thick](0.8,1.5) .. controls (1,0.8) and (0.9,0.7) ..  (0.8,0); 
			\draw [thick](1.6,1.5) .. controls (1.8,0.8) and (	1.7,0.7) ..  (1.6,0); 
			\draw [thick](3.2,1.5) .. controls (3.4,0.8) and (	2.5,0.7) ..  (2.4,0); 
				\draw [thick](4,1.5) .. controls (4.2,0.8) and (3.3,0.7) ..  (3.2,0); 
				
					\draw [thick](0,-0.7) .. controls (0.3,-1.2) and (1.3,-1.2) ..  (1.6,-0.7); 
				
				\node at (-0.8,0.7) {$d$}; 
				\node at (-0.8,-0.7) {$v$};

			\end{tikzpicture}
		\end{array} =\, 
		\begin{array}{c}
			\begin{tikzpicture}
				\draw [green,thick](-0.4,1.5) --(4.4,1.5);
				\draw [green,thick](-0.4,0) --(4.4,0);
				\draw [green,thick](-0.4,1.5) --(-0.4,0);
				\draw [green,thick](4.4,1.5) --(4.4,0);
				\node at (0,1.5) {\textbullet}; 
				\node at (0.8,1.5) {\textbullet};
				\node at (1.6,1.46) {{\huge\textopenbullet}};
				\node at (2.4,1.5) {\textbullet};
				\node at (3.2,1.5) {\textbullet};
				\node at (4,1.46) {{\huge\textopenbullet}};
				\node at (0,-0.04) {{\huge\textopenbullet}}; 
				\node at (0.8,0) {\textbullet};
				\node at (1.6,-0.04) {{\huge\textopenbullet}};
				\node at (2.4,0) {\textbullet};
				\node at (3.2,-0.04) {{\huge\textopenbullet}};
				\node [red]at (4,0) {\textbullet};
				\draw[thick] (0,1.5) .. controls (0.2,1.1) and (0.6,1.1) ..  (0.8,1.5); 
					\draw[thick] (0.8,1.5) .. controls (1,0.8) and (0.9,0.7) ..  (0.8,0); 
							\draw[thick] (1.6,1.5) .. controls (1.8,0.8) and (	1.7,0.7) ..  (1.6,0); 
						\draw [thick](3.2,1.5) .. controls (3.4,0.8) and (	2.5,0.7) ..  (2.4,0); 
						\draw [thick](4,1.5) .. controls (4.2,0.8) and (3.3,0.7) ..  (3.2,0); 
				\draw[thick] (0,0) .. controls (0.3,-0.5) and (1.3,-0.5) ..  (1.6,0); 
				
					\draw [orange,thick](-0.7,1.9) --(4.7,1.9);
				\draw [orange,thick](-0.7,1.1) --(4.7,1.1);
				\draw [orange,thick](-0.7,1.9) --(-0.7,1.1);
				\draw [orange,thick](4.7,1.9) --(4.7,1.1);
				\node [orange,thick]at (5,2.1) {top($d'$)}; 
				\node at (5,0.7) {$v'$}; 
			\end{tikzpicture}
		\end{array}
	\end{equation*}
\begin{equation*}=
	\begin{array}{c}
		\begin{tikzpicture}
			\draw [green,thick](-0.4,0.7) --(4.4,0.7);
			\node at (0,0.7) {\textbullet}; 
			\node at (0.8,0.7) {\textbullet};
			\node at (1.6,0.66) {{\huge\textopenbullet}};
			\node at (2.4,0.7) {\textbullet};
			\node at (3.2,0.7) {\textbullet};
			\node at (4,0.66) {{\huge\textopenbullet}};
			\draw (0,0.7) .. controls (0.2,0.3) and (0.6,0.3) ..  (0.8,0.7); 

			\node [red]at (-0.7,0.7) {$\delta^1$}; 
		\end{tikzpicture}
	\end{array}
\end{equation*}
	\vskip -0.3cm
	\caption{Action on half-diagram  }
	\label{action}
\end{figure}

\begin{figure}[ht]
	\begin{equation*}
		\begin{array}{c}
			\begin{tikzpicture}
				\draw [green,thick](-0.4,1.5) --(4.4,1.5);
				\draw [green,thick](-0.4,0) --(4.4,0);
				\draw [green,thick](-0.4,1.5) --(-0.4,0);
				\draw [green,thick](4.4,1.5) --(4.4,0);
				\draw [green,thick](-0.4,-0.7) --(4.4,-0.7);
				\draw [dotted,thick](0,-0.7) --(0,0);
				\draw [dotted,thick](0.8,-0.7) --(0.8,0);
				\draw [dotted,thick](1.6,-0.7) --(1.6,0);
				\draw [dotted,thick](2.4,-0.7) --(2.4,0);
				\draw [dotted,thick](3.2,-0.7) --(3.2,0);
				\draw [dotted,thick](4,-0.7) --(4,0);
				\node at (0,1.5) {\textbullet}; 
				\node at (0.8,1.5) {\textbullet};
				\node at (1.6,1.5) {\textbullet};
				\node at (2.4,1.5) {\textbullet};
				\node at (3.2,1.5) {\textbullet};
				\node at (4,1.5) {\textbullet};
				\node  at (0,0) {\textbullet}; 
				\node at (0.8,0) {\textbullet};
				\node at (1.6,0) {\textbullet};
				\node at (2.4,0) {\textbullet};
				\node at (3.2,0) {\textbullet};
				\node at (4,0) {\textbullet};
				\node at (0,-0.72) {{\huge\textopenbullet}}; 
				\node at (0.8,-0.7) {\textbullet};
				\node at (1.6,-0.72) {{\huge\textopenbullet}};
				\node at (2.4,-0.7) {\textbullet};
				\node at (3.2,-0.72) {{\huge\textopenbullet}};
				\node at (4,-0.7) {\textbullet};
				\draw [thick](0,1.5) .. controls (0.2,1.1) and (0.6,1.1) ..  (0.8,1.5); 
				\draw [thick](0.8,1.5) .. controls (1,0.8) and (0.9,0.7) ..  (0.8,0); 
				\draw [thick](1.6,1.5) .. controls (1.8,1.1) and (	2.2,1.1) ..  (2.4,1.5); 
				\draw [thick](3.2,1.5) .. controls (3.4,0.8) and (	2.5,0.7) ..  (2.4,0); 
				\draw [thick](4,1.5) .. controls (4.2,0.8) and (3.3,0.7) ..  (3.2,0); 
				
				\draw [thick](0,-0.7) .. controls (0.3,-1.2) and (1.3,-1.2) ..  (1.6,-0.7); 
				
				\node at (-0.8,0.7) {$d$}; 
				\node at (-0.8,-0.7) {$v$};

			\end{tikzpicture}
		\end{array} =\, 
		\begin{array}{c}
			\begin{tikzpicture}
				\draw [green,thick](-0.4,1.5) --(4.4,1.5);
				\draw [green,thick](-0.4,0) --(4.4,0);
				\draw [green,thick](-0.4,1.5) --(-0.4,0);
				\draw [green,thick](4.4,1.5) --(4.4,0);
				\node at (0,1.5) {\textbullet}; 
				\node at (0.8,1.5) {\textbullet};
				\node at (1.6,1.5) {\textbullet};
				\node at (2.4,1.5) {\textbullet};
				\node at (3.2,1.5) {\textbullet};
				\node at (4,1.46) {{\huge\textopenbullet}};
				\node at (0,-0.04) {{\huge\textopenbullet}}; 
				\node at (0.8,0) {\textbullet};
				\node at (1.6,-0.04) {{\huge\textopenbullet}};
				\node at (2.4,0) {\textbullet};
				\node at (3.2,-0.04) {{\huge\textopenbullet}};
				\node at (4,0) {\textbullet};
				\draw[thick] (0,1.5) .. controls (0.2,1.1) and (0.6,1.1) ..  (0.8,1.5); 
				\draw[thick] (0.8,1.5) .. controls (1,0.8) and (0.9,0.7) ..  (0.8,0); 
				\draw [thick](1.6,1.5) .. controls (1.8,1.1) and (	2.2,1.1) ..  (2.4,1.5); 
				\draw [thick](3.2,1.5) .. controls (3.4,0.8) and (	2.5,0.7) ..  (2.4,0); 
				\draw [thick](4,1.5) .. controls (4.2,0.8) and (3.3,0.7) ..  (3.2,0); 
				\draw[thick] (0,0) .. controls (0.3,-0.5) and (1.3,-0.5) ..  (1.6,0); 
				
				\draw [orange,thick](-0.7,1.9) --(4.7,1.9);
				\draw [orange,thick](-0.7,1.1) --(4.7,1.1);
				\draw [orange,thick](-0.7,1.9) --(-0.7,1.1);
				\draw [orange,thick](4.7,1.9) --(4.7,1.1);
				\node [orange,thick]at (5,2.1) {top($d'$)}; 
				\node at (5,0.7) {$v'$}; 
			\node at (2.3,2.7) {$\text{top}(d')$ has fewer labeled blocks than $v$}; 		
			\end{tikzpicture}
		\end{array}
	\end{equation*}
	\begin{equation*}=0
	\end{equation*}
	\vskip -0.3cm
	\caption{Action resulting in $d \cdot v = 0$   }
	\label{action0}
\end{figure}
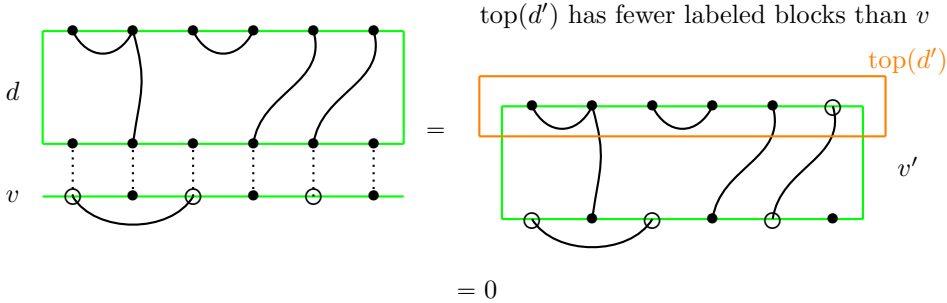

\textbf{Remark 2.1 }The action on half-diagrams satisfies $(d_1 \cdot d_2) \cdot v = d_1 \cdot (d_2 \cdot v)$, where $d_1, d_2$ are partition diagrams and $v$ is a half-diagram. To demonstrate this, we may complete $v$ into a  partition diagram by adding a copy of itself below without strings connecting the upper and lower half-diagrams. The associativity of partition diagram composition then yields the validity of the above equation.

From the decomposition just obtained, the partition algebra can be viewed as the result of a sequence of extensions ---\textit{iterated inflations} in the sense of König and Xi--- of group algebras of symmetric groups along corresponding families of half-diagram modules. 
Inflation theory, developed systematically by Koenig and Xi \cite{KX0,KX1,Xi}, provides an alternative characterisation of cellular algebras, a notion introduced by Graham and Lehrer in their study of the Hecke algebras of symmetric groups \cite{GL}. The symmetric-group algebra, the Temperley–Lieb algebra, the Brauer algebra, and the partition algebra are all cellular. In the cellular framework each of these algebras carries a family of so-called standard (or cell) modules that parametrise its irreducible modules. In the next subsection we review the construction of these standard modules and show that the half-diagram module is precisely the standard module indexed by the one-part partition—that is, by a single-row Young diagram.

\subsection{Standard Modules}

The group algebra $\mathbb{C}\mathfrak{S}_n$ of the symmetric group is semisimple over $\mathbb{C}$ by Maschke's theorem,  so every module decomposes as a direct sum of irreducible modules.  
Classical representation theory shows that the irreducible modules of $\mathbb{C}\mathfrak{S}_n$ are precisely the Specht modules $\Sp(\lambda)$ indexed by partitions $\lambda\vdash n$; these Specht modules also serve as the standard (or cell) modules for $\mathbb{C}\mathfrak{S}_n$ in the cellular algebra sense.

An integer partition of $n$ is a weakly decreasing sequence of positive integers $\lambda = (\lambda_1, \lambda_2, \dots, \lambda_r)$ with  $
\sum_{i=1}^\ell\lambda_i=n$, denoted $\lambda \vdash n$. 
Each summand $\lambda_i$ is called a part.
The Young diagram of $\lambda$ consists of left-aligned rows with $\lambda_i$ boxes. The number of parts $\ell(\lambda)$ equals the number of rows of the associated Young diagram.
A standard Young tableau  is a filling of the boxes with numbers $1, \dots, n$ such that entries increase weakly across rows and strictly down columns. 
Figure \ref{fig:young-diagram-tableau} illustrates the Young diagram of shape $(3,2,1)$ and a standard Young tableau of the same shape. Specht modules are spanned by polytabloids, constructed from Young tableaux via antisymmetrization (see \cite{jk}); in this article we restrict attention to integer‐partition indices.

\begin{figure}[ht]
	\centering
	
	\begin{minipage}[t]{0.45\linewidth}
		\centering
		$\yng(3,2,1)$
		
		\captionof{subfigure}{Young diagram $\lambda=(3,2,1)$}
		\label{fig:young-diagram}
	\end{minipage}
	\hfill
	\begin{minipage}[t]{0.45\linewidth}
		\centering
		$\young(124,35,6)$   
		
		\captionof{subfigure}{A standard Young tableau of shape $(3,2,1)$}
		\label{fig:young-tableau}
	\end{minipage}
	
	\caption{Young diagram and Young tableau}
	\label{fig:young-diagram-tableau}
\end{figure}

Now turning to the partition algebra $P_n$, we let $V(n,r)$ denote the span of all $(n,r)$-partition diagrams with exactly $r$ propagating blocks, endowed with a natural $(P_n(\delta),\mathfrak{S}_r)$-bimodule structure.
As vector spaces, the standard modules are defined by:
\[
\Delta_n(\nu) \cong V(n, r) \otimes_{\mathbb{C}\mathfrak{S}_r} \Sp(\nu),
\]
where $\otimes_{\mathbb{C}\mathfrak{S}_r}$ denotes balanced tensor product over $\mathbb{C}\mathfrak S_r$, i.e. the quotient of the ordinary tensor product by the balanced relations $v\sigma\otimes x = v\otimes\sigma x.$ Figure \ref{Br} depicts the balancing relation diagrammatically.

\begin{figure}[ht]
	\begin{equation*}
		\begin{array}{c}
		\begin{tikzpicture}
			
				\draw [green,thick](-1.9,0.7) --(2.4,0.7);
			
			\draw [green,thick](-0.5,-1) --(1.9,-1);
			\node at (-1.5,0.7) {\textbullet}; 
			\node at (-1,0.7) {\textbullet}; 
			\node at (-0.5,0.7) {\textbullet}; 
			\node at (0,0.7) {\textbullet}; 
			\node at (0.5,0.7) {\textbullet};
			\node at (1,0.7) {\textbullet};
			\node at (1.5,0.7) {\textbullet};
			\node at (2,0.7) {\textbullet};
			\node at (-1.5,1) {\small$1$}; 
			\node at (-1,1) {\small$2$}; 
			\node at (-0.5,1) {\small$3$}; 
			\node at (0,1) {\small$4$}; 
			\node at (0.5,1) {\small$5$};
			\node at (1,1) {\small$6$};
			\node at (1.5,1) {\small$7$};
			\node at (2,1) {\small$8$};
			
			\draw [thick](-1.5,0.7)  .. controls (-1.2,0.4) and (-0.8,0.4) ..  (-0.5,0.7) ; 
			\draw [thick](-1,0.7) .. controls (-0.7,0.45) and (-0.3,0.45) ..  (0,0.7) ;

			\node at (0,-1) {\textbullet}; 
			\node at (0.5,-1) {\textbullet};
			\node at (1,-1) {\textbullet};
			\node at (1.5,-1) {\textbullet};

			\node at (0,-1.3) {$1'$}; 
			\node at (0.5,-1.3) {$2'$};
			\node at (1,-1.3) {$3'$};
			\node at (1.5,-1.3) {$4'$};

			%

			\draw[blue,thick] (-1,0.7)   .. controls (-0.6,0.26) and (1.2,0.5) ..  (1,-1) ; 
			\draw [blue,thick](0.5,0.7)  .. controls (0.35,-0.2) and (0.2,0) ..  (0.5,-1) ; 
			
			\draw [blue,thick](1.5,0.7)   .. controls (1.3,-0.2) and (1.4,0) ..  (1.5,-1) ; 
			\draw [blue,thick](2,0.7)   .. controls (2,-1.2) and (0.2,0) ..  (0,-1) ; 			
				\draw [green,dashed,thick](-1.9,-0.14) --(2.3,-0.14);
							\node at (0.34,-0.14) {\textbullet};
							\node at (0.8,-0.14) {\textbullet};
				\node at (1.38,-0.14) {\textbullet};
				\node at (1.8,-0.14) {\textbullet};
			\node at (2.9,-0.1) {$\o ~ x$};	
			
		\end{tikzpicture}
		\end{array}
		\longleftrightarrow
		\begin{array}{c}
			\begin{tikzpicture}

				\draw [green,thick](-1.9,0.7) --(2.4,0.7);
				
				\draw [green,thick](-0.5,-1) --(1.9,-1);
				\node at (-1.5,0.7) {\textbullet}; 
				\node at (-1,0.7) {\textbullet}; 
				\node at (-0.5,0.7) {\textbullet}; 
				\node at (0,0.7) {\textbullet}; 
				\node at (0.5,0.7) {\textbullet};
				\node at (1,0.7) {\textbullet};
				\node at (1.5,0.7) {\textbullet};
				\node at (2,0.7) {\textbullet};
				\node at (-1.5,1) {\small$1$}; 
				\node at (-1,1) {\small$2$}; 
				\node at (-0.5,1) {\small$3$}; 
				\node at (0,1) {\small$4$}; 
				\node at (0.5,1) {\small$5$};
				\node at (1,1) {\small$6$};
				\node at (1.5,1) {\small$7$};
				\node at (2,1) {\small$8$};
				
				\draw [thick](-1.5,0.7)  .. controls (-1.2,0.4) and (-0.8,0.4) ..  (-0.5,0.7) ; 
				\draw[thick] (-1,0.7) .. controls (-0.7,0.45) and (-0.3,0.45) ..  (0,0.7) ;

				\node at (0,-1) {\textbullet}; 
				\node at (0.5,-1) {\textbullet};
				\node at (1,-1) {\textbullet};
				\node at (1.5,-1) {\textbullet};

				\node at (0,-1.3) {$1'$}; 
				\node at (0.5,-1.3) {$2'$};
				\node at (1,-1.3) {$3'$};
				\node at (1.5,-1.3) {$4'$};

				%

				\draw[blue,thick] (-1,0.7)   .. controls (-0.6,0.26) and (0.5,0.5) ..  (0.5,-1) ; 
				\draw [blue,thick](0.5,0.7)  .. controls (0.5,0.2) and (0.,-0.4) ..  (0.,-1) ; 
				
				\draw [blue,thick](1.5,0.7)   .. controls (1.5,0.2) and (1.,-0.4) ..  (1,-1) ; 
				\draw [blue,thick](2,0.7)   .. controls (2.,0.2) and (1.5,-0.4) ..  (1.5,-1) ; 			
					\node at (3.8,-0.1) {$\o 	\small	\begin{pmatrix}
							1 & 2 & 3 & 4 \\
							2 & 3 &4&1 \\
						\end{pmatrix} \cdot x$};	
				
			\end{tikzpicture}
		\end{array}
	\end{equation*}
	\vskip -0.3cm
	\caption{Balanced relation}
	\label{Br}
\end{figure}
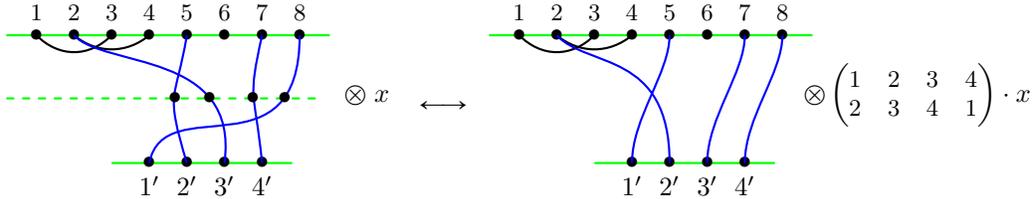

The module action is defined below.
For $v \in V(n, r)$, $x \in \Sp(\nu)$, and $d \in P_n(\delta)$, the action $d \cdot (v \otimes x)$ is defined via diagram concatenation:
Concatenate $d$ and $v$ to obtain $\delta^t v'$, where $v'$ is an $(n, r)$-diagram and $t$ is the number of closed loops.
 If $v'$ has fewer than $r$ propagating blocks, set $d \cdot (v \otimes x) = 0$.
Otherwise, let
$d \cdot (v \otimes x) = \delta^t v' \otimes  x.$

 Denote by $\Lambda_n$ the set of all partitions of $n$.
Let $\Lambda_{\leq n} = \bigcup_{0 \leq i \leq n} \Lambda_i$. For $\delta \notin \{0, 1, \dots, 2n - 2\}$, the algebra $P_n(\delta)$ is semisimple, and $\{\Delta_n(\nu) \mid \nu \in \Lambda_{\leq n}\}$ forms a complete set of non-isomorphic  irreducible modules.

Later we will encounter the case of one-part partitions.
When $\nu = (r)$, the Specht module $\Sp((r))$ is the trivial $\mathfrak{S}_r$-module $\mathbbm{1}_r$. The standard module then reduces to the half-diagram module:
$$
\Delta_n((r))
= V(n,r)\otimes_{\mathbb{C}\mathfrak{S}_r}\mathbbm{1}_r\;\cong\;V_n(r)\;,
$$
 Figure \ref{HO} illustrates this isomorphism.

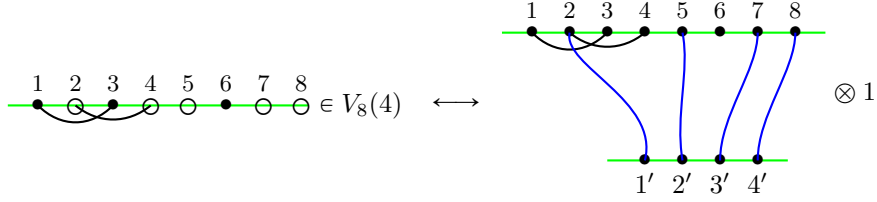
\begin{figure}[ht]
	\begin{equation*}
		\begin{array}{c}
			\begin{tikzpicture}
			
			\draw [green,thick](-1.9,0.7) --(2.1,0.7);

			\node at (2.8,0.7) {$\in V_{8}(4) $}; 
			
			\node at (-1.5,0.7) {\textbullet}; 
			\node at (-1,0.66) {{\huge\textopenbullet}}; 
			\node at (-0.5,0.7) {\textbullet}; 
			\node at (0,0.66) {{\huge\textopenbullet}}; 
			\node at (0.5,0.66) {{\huge\textopenbullet}};
			\node at (1,0.7) {\textbullet};
			\node at (1.5,0.66) {{\huge\textopenbullet}};
			\node at (2,0.66) {{\huge\textopenbullet}};
			\node at (-1.5,1) {\small$1$}; 
			\node at (-1,1) {\small$2$}; 
			\node at (-0.5,1) {\small$3$}; 
			\node at (0,1) {\small$4$}; 
			\node at (0.5,1) {\small$5$};
			\node at (1,1) {\small$6$};
			\node at (1.5,1) {\small$7$};
			\node at (2,1) {\small$8$};

			\draw [thick](-1.5,0.7)  .. controls (-1.2,0.4) and (-0.8,0.4) ..  (-0.5,0.7) ; 
			\draw[thick] (-1,0.7) .. controls (-0.7,0.45) and (-0.3,0.45) ..  (0,0.7) ; 

		\end{tikzpicture} 
		\end{array}
		\longleftrightarrow
		\begin{array}{c}
			\begin{tikzpicture}

				\draw [green,thick](-1.9,0.7) --(2.4,0.7);
				
				\draw [green,thick](-0.5,-1) --(1.9,-1);
				\node at (-1.5,0.7) {\textbullet}; 
				\node at (-1,0.7) {\textbullet}; 
				\node at (-0.5,0.7) {\textbullet}; 
				\node at (0,0.7) {\textbullet}; 
				\node at (0.5,0.7) {\textbullet};
				\node at (1,0.7) {\textbullet};
				\node at (1.5,0.7) {\textbullet};
				\node at (2,0.7) {\textbullet};
				\node at (-1.5,1) {\small$1$}; 
				\node at (-1,1) {\small$2$}; 
				\node at (-0.5,1) {\small$3$}; 
				\node at (0,1) {\small$4$}; 
				\node at (0.5,1) {\small$5$};
				\node at (1,1) {\small$6$};
				\node at (1.5,1) {\small$7$};
				\node at (2,1) {\small$8$};
				
				\draw [thick](-1.5,0.7)  .. controls (-1.2,0.4) and (-0.8,0.4) ..  (-0.5,0.7) ; 
				\draw[thick] (-1,0.7) .. controls (-0.7,0.45) and (-0.3,0.45) ..  (0,0.7) ;

				\node at (0,-1) {\textbullet}; 
				\node at (0.5,-1) {\textbullet};
				\node at (1,-1) {\textbullet};
				\node at (1.5,-1) {\textbullet};

				\node at (0,-1.3) {$1'$}; 
				\node at (0.5,-1.3) {$2'$};
				\node at (1,-1.3) {$3'$};
				\node at (1.5,-1.3) {$4'$};

				%

				\draw[blue,thick] (-1,0.7)   .. controls (-1.1,0.26) and (0.2,-0.2) ..  (0,-1) ; 
				
				\draw [blue,thick](0.5,0.7)  .. controls (0.6,0.2) and (0.4,-0.4) ..  (0.5,-1) ; 
				
				\draw [blue,thick](1.5,0.7)   .. controls (1.5,0.2) and (1.,-0.4) ..  (1,-1) ; 
				\draw [blue,thick](2,0.7)   .. controls (2.,0.2) and (1.5,-0.4) ..  (1.5,-1) ; 			
				\node at (2.8,-0.1) {$\o \; 1$};	
			\end{tikzpicture}
		\end{array}
	\end{equation*}

	\vskip -0.3cm
	\caption{Half-diagram module as one-part partitions standard module}
	\label{HO}
\end{figure}

To conclude this section, we generalize \cite[Proposition 3]{HP1}
from the Brauer algebra to the partition algebra.

\begin{prop}\label{aalc} Keep the notations as above.
	Let $M$ and $N$ be $ \mathbb{C}\mathfrak{S}_{r}$-modules, and let  $V(n, r)$ be the module of all $(n, r)$-diagrams having precisely $r$ propagating blocks of $P_{n}$. Then $$
	\Hom_{P_{n}} (V(n, r)\o_{ \mathbb{C}\mathfrak{S}_{r}} M, V(n, r)\o_{ \mathbb{C}\mathfrak{S}} N)\cong \Hom_{ \mathbb{C}\mathfrak{S}_{r}} (M, N).$$

\end{prop}
\begin{proof}
	Denote by $u_0 \in V(n, r)$ the $(n, r)$-partition diagram
	\begin{figure}[H]
		\begin{equation*}u_{0}= 
			\begin{array}{c}
			\begin{tikzpicture}
				\node at (-1.6,0.9) {\tiny\textbullet};
				\node at (-1.6,0) {\tiny\textbullet};
				\draw (-1.6,0.9) --(-1.6,0);
				
				\node at (-0.4,0.9) {\tiny\textbullet};
				\node at (-0.4,0) {\tiny\textbullet};
				
				\node at (-0.8,0.9) {$\cdots$};
				\node at (-0.8,0) {$\cdots$};
				\node at (-0.77,-0.2) {\scriptsize$r$};
				\draw (-0.4,0) --(-0.4,0.9);
				\node at (0,0.9) {\tiny\textbullet};
				\node at (0,0) {\tiny\textbullet};
				\node at (-1.2,0.9) {\tiny\textbullet};
				\node at (-1.2,0) {\tiny\textbullet};
				\draw (-1.2,0) --(-1.2,0.9);
				
				\node at (0.8,0.9) {\tiny\textbullet};

				
				\node at (1.6,0.9) {\tiny\textbullet};

				\node at (0.4,0.9) {\tiny\textbullet};
				\draw (0,0.9) .. controls (0.1,0.7) and (0.3,0.7) .. (0.4,0.9);
			
				\draw (0.4,0.9) .. controls (0.5,0.7) and (0.7,0.7) .. (0.8,0.9);

				\node at (1.2,0.9) {$\cdots$};
			
				\draw (0,0.9)--(0, 0);

			\end{tikzpicture}
			\end{array}
		\end{equation*}
		\vskip -0.3cm
	\end{figure}
	and define an idempotent $e_{r}$ to be
	\begin{equation*} e_{r}= 
		\begin{array}{c}
			\begin{tikzpicture}
					\node at (-1.6,0.9) {\tiny\textbullet};
						\node at (-1.6,0) {\tiny\textbullet};
							\draw (-1.6,0.9) --(-1.6,0);
							
				\node at (-0.4,0.9) {\tiny\textbullet};
				\node at (-0.4,0) {\tiny\textbullet};
			
				\node at (-0.8,0.9) {$\cdots$};
				\node at (-0.8,0) {$\cdots$};
				\node at (-0.9,1.1) {\scriptsize$r-1$};
				\draw (-0.4,0) --(-0.4,0.9);
				\node at (0,0.9) {\tiny\textbullet};
				\node at (0,0) {\tiny\textbullet};
				\node at (-1.2,0.9) {\tiny\textbullet};
				\node at (-1.2,0) {\tiny\textbullet};
				\draw (-1.2,0) --(-1.2,0.9);
		
				\node at (0.8,0.9) {\tiny\textbullet};
				\node at (0.8,0) {\tiny\textbullet};
			
			
				\node at (1.6,0.9) {\tiny\textbullet};
				\node at (1.6,0) {\tiny\textbullet};
				\node at (1.6,0) {\tiny\textbullet};
							\node at (0.4,0) {\tiny\textbullet};
								\node at (0.4,0.9) {\tiny\textbullet};
				\draw (0,0.9) .. controls (0.1,0.7) and (0.3,0.7) .. (0.4,0.9);
					\draw (0,0) .. controls (0.1,0.2) and (0.3,0.2) .. (0.4,0);
					
					\draw (0.4,0.9) .. controls (0.5,0.7) and (0.7,0.7) .. (0.8,0.9);
					\draw (0.4,0) .. controls (0.5,0.2) and (0.7,0.2) .. (0.8,0);
				
				\node at (1.2,0.9) {$\cdots$};
				\node at (1.2,0) {$\cdots$};

		\draw (0,0.9)--(0, 0);

			\end{tikzpicture}
		\end{array}
	\end{equation*}
	Here, note that $\delta \neq 0$ by assumption.
	
	Let $\phi \in\Hom_{P_{n}} (V(n, r)\o_{ \mathbb{C}\mathfrak{S}_{r}} M, V(n, r)\o_{ \mathbb{C}\mathfrak{S}_{r}} N)$. For arbitrary $m\in M$, assume that
	$$\phi( u_0 \o m)= \sum_{i\in I}   v_i\o n_i.$$ Apply the idempotent $e_{r}$
	on both sides of the equality above. Note that $$e_{r}( u_0\o m)=   u_0\o m.$$ Moreover, each $e_{r}\cdot v_i$
	is either zero or a multiple of $u_0$ and thus there exists certain $n\in N$ such that $$e_{r}.\sum v_i\otimes n_i= u_0\o n.$$
	Consequently, we obtain a map $\hat{\phi}(m)=n$ from $M$ to $N$ induced by $\phi$.
	
	Moreover, for each $v \in  V(n, r)$ there exists an element
	$a \in P_n$ such that $v = a u_0 $.  Therefore, for each $m\in M$  we have
$$
		\phi( v \o m)	=  \phi(a u_0\o m)=a\phi( u_0\o m)\\
		 = a( u_0\o\hat{\phi}(m)))=v\o \hat{\phi}(m).
$$
	
	To show  \(\mathfrak{S}_r\)-linearity of \(\hat{\phi}\): 
For any \( s \in \mathfrak{S}_r \), there exists \( p \in P_n \) such that \( u_0 s = p u_0 \) (achieved by permuting propagating lines via \( p \)). Then
\[
\phi(u_0 \otimes s m) = \phi(p u_0 \otimes m) = p \cdot \phi(u_0 \otimes m) = p \cdot (u_0 \otimes \hat{\phi}(m)) = u_0 s \otimes \hat{\phi}(m) = u_0 \otimes s \hat{\phi}(m),  
\]  
which implies \( \hat{\phi}(s m) = s \hat{\phi}(m) \). Thus, \( \hat{\phi} \in \Hom_{\mathfrak{S}_r}(M, N) \).

	Conversely, for any $v\in V(n,r), m \in M$, and   homomorphism $\hat{\phi}\in \Hom_{\mathbb{C}\mathfrak{S}_{r}} (M, N)$,
	define $\phi( v\o m)= v \o \hat{\phi}(m)$.
To verify \( \phi \) is a \( P_n \)-module homomorphism, observe that for any \( d \in P_n \):  
\[
\phi(d \cdot (v \otimes m)) = \phi(d v \otimes m) = dv\o \hat{\phi}(m)= d \cdot \phi(v \otimes m).  
\]

\end{proof}

\textbf{Assumption.} We assume throughout the paper that $\delta \not\in \{0,1,\ldots, 2n-2\}$, saying that $P_n(\delta)$ is semisimple.

\section{Multiplicities}
In this section, we specialize the formula of Bowman, De Visscher, and Orellana to the case of one-part partitions and reduce the multiplicity computation to solving the Diophantine system (E1).

\subsection{The Restriction of a standard module}
The Young subgroup \(\mathfrak{S}_m \times \mathfrak{S}_n \subset \mathfrak{S}_{m+n}\) (also denoted \(\mathfrak{S}_{m,n}\)) is the subgroup that independently permutes the first \(m\) and the last \(n\) letters.  For $\lambda \vdash m$ and $\mu \vdash n$, the outer tensor product $\Sp(\lambda) \boxtimes \Sp(\mu)$ is defined as the $\mathfrak{S}_m \times \mathfrak{S}_n$-module with the action $(\sigma, \tau) \cdot (v \otimes w) = (\sigma \cdot v) \otimes (\tau \cdot w)$. 

When  restricting a Specht module for $\mathfrak S_{m+n}$ to this Young subgroup $\mathfrak{S}_m \times \mathfrak{S}_n$, it breaks into outer tensor products according to the classical Littlewood–Richardson rule:

\begin{thm}[Littlewood–Richardson Rule]
	Let $\lambda \vdash m$, $\mu \vdash n$, and $\nu \vdash m+n$. The restriction of the Specht module $\Sp(\nu)$ to $\mathfrak{S}_m \times \mathfrak{S}_n$ decomposes as  
	\[
	\Sp(\nu) \!\!\downarrow_{\mathfrak{S}_m \times \mathfrak{S}_n} \cong \bigoplus_{\lambda, \mu} c^\nu_{\lambda, \mu} \, \Sp(\lambda) \boxtimes \Sp(\mu),
	\]  
	where the multiplicities $c^\nu_{\lambda, \mu}$ are Littlewood–Richardson coefficients (L-R coefficients).  
\end{thm}

Equivalently, the Littlewood–Richardson coefficient $c_{\lambda,\mu}^{\nu}$ is the multiplicity with which the outer tensor product $\Sp(\lambda)\boxtimes\Sp(\mu)$ appears as a composition factor of $\Sp(\nu)$; this is recorded by the bracket
$$
\bigl[\Sp_{m+n}(\nu)\!\downarrow_{\mathfrak{S}_m\times\mathfrak{S}_n} : \Sp_m(\lambda)\boxtimes\Sp_n(\mu)\bigr].
$$
An analogous formula for induction is obtained via Frobenius reciprocity.

In the partition algebra case, an analogous restriction theorem was obtained by Bowman, De Visscher and Orellana in \cite{BVC}. Let $\Delta_m(\lambda)$ and $\Delta_n(\mu)$ be standard modules of partition algebras $P_m$ and $P_n$, respectively, and $\Delta_{m+n}(\nu)$ the standard module of $P_{m+n}$. The algebra $P_m \otimes P_n$ is the tensor product of partition  algebras with component-wise multiplication $(a \otimes b)(a' \otimes b') = (aa') \otimes (bb')$.  
Considering the action of the algebra $P_m \otimes P_n$ on the outer tensor product $\Delta_m(\lambda)\boxtimes \Delta_n(\mu)$, the restriction structure is described by the three-part Littlewood-Richardson coefficient and Kronecker coefficient:

\begin{thm}[\cite{BVC}, Theorem 4.3] \label{par-mul}
	Let $\nu\vdash m+n-l$, $\lambda\vdash m-l_m$ and $\mu\vdash n-l_n$ for some non-negative integers $l, l_m, l_n$. Then $\Delta_m(\nu)\!\!\downarrow_{P_r\otimes P_s}$ has  a filtration by standard modules with multiplicities given by
	$$[\Delta_{m+n}(\nu)\!\downarrow_{P_m\otimes P_n} \, : \, \Delta_m(\lambda)\boxtimes \Delta_n(\mu)]=\sum_{\begin{subarray}{c} l_1, l_2 \\ l_1+2l_2=l-l_m-l_n \end{subarray}} \sum_{\begin{subarray}{c} \alpha\vdash m-l_m-l_1-l_2 \\ \beta\vdash n-l_n-l_1-l_2 \\ \pi,\rho,\sigma\vdash l_1 \\ \gamma\vdash l_2 \end{subarray}} c_{\alpha ,\beta ,\pi}^\nu c_{\alpha , \rho ,\gamma}^\lambda c_{\gamma ,\sigma , \beta}^{\mu}g_{\rho,\sigma}^\pi .$$
\end{thm} 

Here, $c^\nu_{\lambda,\mu,\eta}=\sum_{\xi}c^\xi_{\lambda,\mu} c^\nu_{\xi,\eta}$ is the three-part Littlewood-Richardson coefficient,  defined as the sum of multiplicities of $\Sp(\lambda)\boxtimes\Sp(\mu)\boxtimes\Sp(\eta)$ in $\Sp(\nu)\!\! \downarrow_{\mathfrak{S}_{r_1} \times \mathfrak{S}_{r_2}\times\mathfrak{S}_{r_3} }^{\mathfrak{S}_{r_1+r_2+r_3}}$, while the Kronecker coefficient $g_{\lambda , \mu}^\nu$ corresponds to the multiplicity in the direct sum decomposition $\Sp(\lambda)\otimes \Sp(\mu) = \bigoplus_{\nu} g_{\lambda,\mu}^\nu \Sp(\nu)$.

\subsection{L-R and Kronecker Coefficients for One-Part Partitions}
First, we recall how to compute Littlewood–Richardson coefficients using the combinatorial algorithm of Young tableaux; see \cite[Section 2.8.13]{jk} for details.

The Littlewood–Richardson coefficient $c^\nu_{\lambda,\mu}$ is zero unless $\lambda \subseteq \nu$. If $\lambda \subseteq \nu$, it can be calculated as follows:
For each node $(i,j)$ of $\mu$, take a symbol $u_{ij}$.  Begin with the diagram $\lambda$ and:
\begin{enumerate}[leftmargin=*,itemsep=0.1em]
	\item Add to it all symbols $u_{1j}$ (corresponding to the first row of nodes of $\mu$) in such a way as to produce the diagram of a partition and to satisfy (3).
	\item Next add all symbols $u_{2j}$ (corresponding to the second row of nodes of $\mu$) following the same rules.  Continue this process with all rows of $\mu$.
	\item The added symbols must satisfy: $(a)$ for all $i$, if $y <j$, $u_{iy}$ is in a later column than $u_{ij}$; and $(b)$ for all $j$, if $x<i$ $u_{xj}$ is in an earlier row than $u_{ij}$.
\end{enumerate}

For example, take $\lambda = (2)$ and $\mu = (2,1)$.  Applying the algorithm above produces exactly the cases shown in Fig.\ref{fig:LR}, with each shape appearing exactly once.  Hence,

$$
c^{(4,1)}_{(2),(2,1)} \;=\;
c^{(3,2)}_{(2),(2,1)} \;=\;
c^{(3,1,1)}_{(2),(2,1)} \;=\;
c^{(2,2,1)}_{(2),(2,1)} \;=\; 1,
\quad
\text{and all other } c^{\nu}_{(2),(2,1)} = 0.
$$

\begin{figure}[ht]
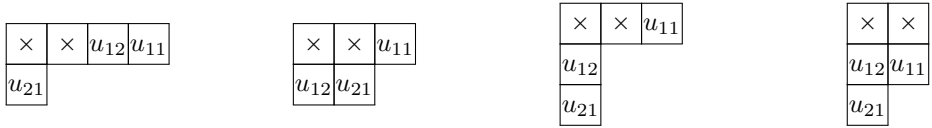

	\centering
	\begin{minipage}{0.24\textwidth}
		\centering
		\begin{ytableau}
			\times & \times & u_{12} & u_{11}\\
			u_{21}
		\end{ytableau}
	\end{minipage}
	\hfill
	\begin{minipage}{0.24\textwidth}
		\centering
		\begin{ytableau}
				\times  & 	\times  & u_{11} \\
			u_{12}  & u_{21} 
		\end{ytableau}
	\end{minipage}
	\hfill
	\begin{minipage}{0.24\textwidth}
		\centering
		\begin{ytableau}
				\times  & 	\times  &u_{11} \\
			u_{12}  \\
			u_{21} 
		\end{ytableau}
	\end{minipage}
	\hfill
	\begin{minipage}{0.24\textwidth}
		\centering
		\begin{ytableau}
				\times  & 	\times  \\
			u_{12}  & u_{11} \\
			u_{21} 
		\end{ytableau}
	\end{minipage}
	\caption{An example for $\lambda=(2)$,$\mu=(2,1)$}
	\label{fig:LR}
\end{figure}

\begin{lemma}\label{two-part}
	Let $\nu=(r)$ be a one-part partition.  Then the Littlewood–Richardson coefficient $c^{(r)}_{\lambda,\mu}$ satisfies
	\[
	c^{(r)}_{\lambda,\mu}=
	\begin{cases}
		1,&\text{if } \lambda=(k),\ \mu=(r-k) \text{ for some } 0\leq k\leq r,\\
		0,&\text{otherwise}.
	\end{cases}
	\]
\end{lemma}

\begin{proof}
By the Littlewood–Richardson rule, the composite Young tableau of $\lambda$ and $\mu$ must satisfy row non-decreasing and column increasing conditions. If the Young diagram of $\lambda$ or $\mu$ has multiple rows, their symbol fillings would introduce additional rows in the composite diagram, contradicting the single-row structure of the Young diagram of shape $\nu = (r)$. Thus, $\lambda$ and $\mu$ must both be one-part partitions, with $\lambda = (k)$ and $\mu = (r - k)$ satisfying $k + (r - k) = r$. Appending symbols of $\mu$ to the right of $\lambda$ yields the unique valid filling, giving $c^{(r)}_{\lambda, \mu} = 1$. Otherwise, the coefficient vanishes.
\end{proof}

\begin{lemma}\label{three-part}
	Let $\nu = (r)$ be a one-part partition. The three-part Littlewood–Richardson coefficient $c^{(r)}_{\lambda,\mu,\eta}$ is non-zero if and only if there exist non-negative integers $k, m \geq 0$ satisfying $k + m \leq r$, such that:
	\[
	\lambda = (k),\quad \mu = (m),\quad \eta = (r - k - m).
	\]
	In this case, $c^{(r)}_{\lambda,\mu,\eta} = 1$; otherwise, the coefficient vanishes.

\end{lemma}

\begin{proof}
The three-part coefficient is defined as $c^{(r)}_{\lambda, \mu, \eta} = \sum_{\xi} c^{\xi}_{\lambda, \mu} c^{(r)}_{\xi, \eta}$. By Lemma 1, if $\xi$ is not one-part, both $c^{\xi}_{\lambda, \mu}$ and $c^{(r)}_{\xi, \eta}$ vanish. The only non-zero term corresponds to $\xi = (k + m)$, where $\lambda = (k)$ and $\mu = (m)$ are one-part. Further, $c^{(r)}_{(k + m), \eta} \neq 0$ requires $\eta = (r - k - m)$ to be one-part. Each factor contributes a coefficient of 1, hence the total is 1.
\end{proof}

\begin{lemma} Let $\lambda = \mu = \nu = (n)$. Then the Kronecker coefficient satisfies  
	\[
	g_{(n), (n)}^{(n)} = 1.
	\]
\end{lemma}

\begin{proof}
The one-part  partition $(n)$ corresponds to the trivial module of $\mathfrak{S}_n$, where every group element acts as the identity. The tensor product $\Sp((n)) \otimes \Sp((n))$ remains trivial, decomposing into a single copy of $\Sp((n))$. Thus, $g^{(n)}_{(n), (n)} = 1$.
\end{proof}

\subsection{Diophantine equation}
When the indices in Theorem \ref{par-mul} take the one-part partitions $\nu=(r)$, $\lambda=(p)$, and $\mu=(q)$, we define  
$$E_{p, q}^{r} = \left[\Delta_{m+n}\big((r)\big)\!\!\downarrow_{P_m \otimes P_n} \, : \, \Delta_m\big((p)\big) \boxtimes \Delta_n\big((q)\big)\right].$$

\begin{prop}
	The multiplicity $E_{p, q}^{r}$ equals the number of non-negative integer solutions to the system (E1):
	\[
	\left\{
	\begin{array}{ll}
		T+ L + R = r, \\
		L+ T + U= p, \\
		R + U+T  = q,
	\end{array} 
	\right.\eqno(\mathrm{E1})
	\]
\end{prop}

\begin{proof}
By the Bowman–De Visscher–Orellana formula given in Theorem \ref{par-mul}, set $r = m + n - l$, $p = m - l_m$, $q = n - l_n$, and replace $l_1$ and $l_2$ with $T$ and $U$, respectively. Then $E_{p, q}^{r}$ is expressed as:
\[
E_{p, q}^{r} = \sum_{\substack{T, U \\ T + 2U = p + q - r}} \sum_{\substack{\alpha \vdash L, \, \beta \vdash R \\ \pi, \rho, \sigma \vdash T \\ \gamma \vdash U}} c^{(r)}_{\alpha, \beta, \pi} c^{(p)}_{\alpha, \rho, \gamma} c^{(q)}_{\gamma, \sigma, \beta} g^{\pi}_{\rho, \sigma},
\]
where $L = p - U - T$ and $R = q - U - T$. 

From Lemmas 1-3, non-zero terms exist only when all partitions are one-part:  

1. $c^{(r)}_{\alpha, \beta, \pi} = 1$ requires $\alpha = (L)$, $\beta = (R)$, $\pi = (T)$, and $L + R + T = r$. 
 
2. $c^{(p)}_{\alpha, \rho, \gamma} = 1$ requires $\alpha = (L)$, $\rho = (T)$, $\gamma = (U)$, and $L + T + U = p$.  

3. $c^{(q)}_{\gamma, \sigma, \beta} = 1$ requires $\gamma = (U)$, $\sigma = (T)$, $\beta = (R)$, and $U + T + R = q$.  

4. $g^{\pi}_{\rho, \sigma} = 1$ requires $\pi = \rho = \sigma = (T)$.  

These conditions are equivalent to the system (E1),
whose non-negative integer solutions $(T, U, L, R)$ correspond bijectively to non-zero terms in the sum (each contributing 1). Thus, $E_{p, q}^{r}$ counts the number of solutions.

\end{proof}

\section{Diagram Algebra Approach: A Half-Diagram Decomposition Perspective}
In the work of Bowman et al. \cite{BVC}, restriction multiplicities are studied through diagram decomposition via inflation. The motivation for this section is to adapt their approach by directly decomposing half-diagrams, aiming both to clarify the meaning of the unknowns in the Diophantine equation (E1) and to explain why each solution of the equation contributes exactly 1 to the multiplicity when used as index values.

The tensor product of two diagrams \(d_1 \otimes d_2\) is defined as their horizontal juxtaposition, where \(d_1\) is placed to the left of \(d_2\) (see Figure \ref{walled}(1) for an illustration). The collection of all such juxtaposed diagrams forms a basis for the tensor product algebra \(P_m \otimes P_n\). Similarly, for two half-diagrams \(v_1\) and \(v_2\), the tensor product \(v_1 \otimes v_2\) corresponds to placing \(v_1\) horizontally to the left of \(v_2\), thereby generating a basis for the \(P_m \otimes P_n\)-module \(V_{m}(p) \boxtimes V_{n}(q)\), where \(P_m\) and \(P_n\) act independently on \(V_{m}(p)\) and \(V_{n}(q)\), respectively.

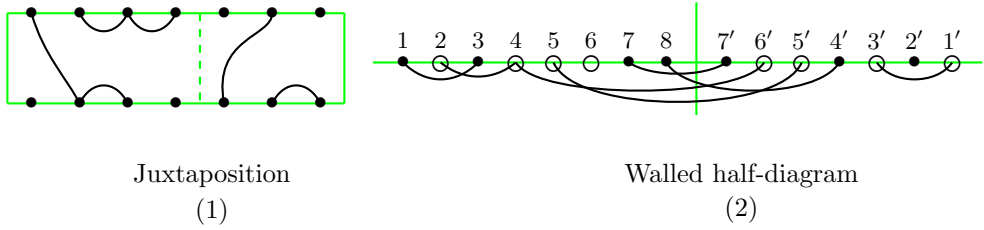
\begin{figure}[ht] \label{W}
	\begin{equation*}
		\begin{array}{c}
			\begin{tikzpicture}[scale=0.8] 
				\draw [green,thick](-0.4,1.5) --(2.8,1.5);
				\draw [green,thick](-0.4,0) --(2.8,0);
				\draw [green,thick](-0.4,1.5) --(-0.4,0);
				\draw [green,dashed,thick](2.8,1.5) --(2.8,0);
			
				\node at (0,1.5) {\textbullet}; 
				\node at (0.8,1.5) {\textbullet};
				\node at (1.6,1.5) {\textbullet};
				\node at (2.4,1.5) {\textbullet};
				\node at (0,0) {\textbullet}; 
				\node at (0.8,0) {\textbullet};
				\node at (1.6,0) {\textbullet};
				\node at (2.4,0) {\textbullet};
				\draw[thick] (0.8,1.5) .. controls (1,1.1) and (1.4,1.1) ..  (1.6,1.5); 
				\draw[thick] (0,1.5) .. controls (0.2,0.8) and (0.6,0.4) ..  (0.8,0); 
				\draw[thick] (0.8,0) .. controls (1,0.4) and (1.4,0.4) ..  (1.6,0); 
				\draw [thick](1.6,1.5) .. controls (1.8,1.1) and (2.2,1.1) ..  (2.4,1.5);
				\node at (3,-1.2) { Juxtaposition};
				\draw [green,thick](2.8,1.5) --(5.2,1.5); 
				\draw [green,thick](2.8,0) --(5.2,0);

				\draw [green,thick](5.2,1.5) --(5.2,0);

				\node at (3.2,1.5) {\textbullet}; 
				\node at (4.0,1.5) {\textbullet};
				\node at (4.8,1.5) {\textbullet};

				\node at (3.2,0) {\textbullet}; 
				\node at (4,0) {\textbullet};
				\node at (4.8,0) {\textbullet};
	\draw [thick](4.0,1.5) .. controls (4,1.1) and (3,1.1) ..  ((3.2,0);
						\draw [thick](4,0) .. controls (4.2,0.4) and (4.6,0.4) ..  (4.8,0); 
								\node at (3,-1.8) {(1) };
			\end{tikzpicture}
		\end{array}
		\begin{array}{c}
		\begin{tikzpicture}
		\draw [green,thick](-1.9,0.7) --(6.2,0.7);
		\draw [green,thick](2.4,1.5) --(2.4,0);

			\node at (-1.5,0.7) {\textbullet}; 
			\node at (-1,0.67) {{\huge\textopenbullet}}; 
		\node at (-0.5,0.7) {\textbullet}; 
		\node at (0,0.67) {{\huge\textopenbullet}}; 
		\node at (0.5,0.67) {{\huge\textopenbullet}};
		\node at (1,0.67) {{\huge\textopenbullet}};
		\node at (1.5,0.7) {\textbullet};
		\node at (2,0.7) {\textbullet};
			\node at (-1.5,1) {\small$1$}; 
		\node at (-1,1) {\small$2$}; 
		\node at (-0.5,1) {\small$3$}; 
					\node at (0,1) {\small$4$}; 
	\node at (0.5,1) {\small$5$};
\node at (1,1) {\small$6$};
\node at (1.5,1) {\small$7$};
\node at (2,1) {\small$8$};

				\node at (2.8,0.7) {\textbullet};
						\node at (3.3,0.67) {{\huge\textopenbullet}};
				\node at (3.8,0.67) {{\huge\textopenbullet}};
						\node at (4.3,0.7) {\textbullet};
		\node at (4.8,0.67) {{\huge\textopenbullet}};
				\node at (5.3,0.7) {\textbullet};
						\node at (5.8,0.67) {{\huge\textopenbullet}};
						
							\node at (2.8,1) {\small$7'$};
						\node at (3.3,1) {\small$6'$};
						\node at (3.8,1) {\small$5'$};
						\node at (4.3,1) {\small$4'$};
						\node at (4.8,1) {\small$3'$};
						\node at (5.3,1) {\small$2'$};
						\node at (5.8,1) {\small$1'$};
			\draw [thick](-1.5,0.7)  .. controls (-1.2,0.4) and (-0.8,0.4) ..  (-0.5,0.7) ; 
		\draw [thick](-1,0.7) .. controls (-0.7,0.45) and (-0.3,0.45) ..  (0,0.7) ; 
			\draw [thick](0,0.7) .. controls (0.3,0.2) and (3,0.2) ..  (3.3,0.7) ; 
				\draw [thick](0.5,0.7)   .. controls (0.8,0) and (3.5,0) ..  (3.8,0.7) ; 
					\draw [thick](1.5,0.7)  .. controls (1.8,0.5) and (2.5,0.5) ..  (2.8,0.7) ; 
					\draw [thick](2,0.7)  .. controls (2.3,0.2) and (4,0.2) ..  (4.3,0.7) ; 
						\draw [thick](4.8,0.7)  .. controls (5,0.4) and (5.6,0.4) ..  (5.8,0.7) ; 
			\node at (3,-1.25) {(2) };
		\node at (3,-0.8) { Walled half-diagram};
	\end{tikzpicture}
	\end{array}
	\end{equation*}
	\vskip -0.3cm
	\caption{Walled diagram }
	\label{walled}
\end{figure}

\subsection{Walled Diagrams and Parameterization} 
When we consider the action of \(P_m \otimes P_n\) on an \((m+n)\)-half-diagram, a vertical line can be added to the diagram as a `wall' (see Figure \ref{walled}(2) for an example). 
Formally, an $(m~|~n, r)$-\textit{walled half-diagram} is a set-partition of the set $\{1,\cdots,m, n',\cdots,1'\}$ with ordering $1<2<\cdots<m<n'<\cdots<1'$, and has $r$ labeled blocks. When represented by a diagram, a vertical line is added between dot $m$ and dot $n'$ as a wall, and all dots in labeled blocks are marked with white dots. 

In the proof of Theorem 4.3 in \cite{BVC}, Bowman et al. introduced certain indices when decomposing \((n, r)\)-diagrams. These same indices remain applicable for decomposing half-diagrams, and we adopt them here following their approach.

Blocks of a walled half-diagram can be divided into the following types:

1) Labeled blocks crossing the wall are referred to as \textit{through-labeled blocks}, and $T$ denotes the number of such blocks.

2) Unlabeled blocks crossing the wall are referred to as \textit{through-unlabeled blocks}, and $U$ denotes the number of such blocks.

3) The labeled blocks on the left (or right) sides of the wall are referred to as the \textit{left (or right) labeled blocks}, and the number is denoted as $L$ (or $R$). 

There are also unlabeled blocks on the left or right. Because the count of such blocks is not needed as a key index in the restriction to $P_m \otimes P_n$, we do not introduce special notation for it.

For each walled half-diagram, we assign an index $(U~;~T, L, R)$, where the right side of the semicolon indicates the numbers of the three types of labeled blocks.
Obviously, $T+L+R=r$. Figure \ref{walled}(2) shows an example of a walled half diagram, with $\{2,4,6'\}$, $\{5,5'\}$ being  through-labeled blocks, $\{7,7'\}$,  $\{8,4'\}$  being  through-unlabeled blocks, and $\{6\}$ and $\{3', 1'\}$ being the left and right labeled blocks, respectively. Figure \ref{walled}(2) is thus a $(8~|~7,~4)$-walled half-diagram with the index $(2~;~2,1,1)$.

When considering the restriction of \(P_{m+n}\) to \(P_m \otimes P_n\), the changes in indices under the action on walled diagrams can be summarized in Lemma \ref{order}. The conclusions can be directly derived by considering the action of the generators of \(P_m \otimes P_n\) on walled diagrams. Here, the generators are obtained by juxtaposing the identity diagrams $\begin{minipage}[c]{4cm} \xysmall{\bullet \tra[d] & \bullet \tra[d]& \cdots &\bullet \tra[d]   \\ \bullet& \bullet & \cdots& \bullet} \end{minipage}$ on the left and right of the three types of diagrams \(e_{i,j}\), \(p_i\), \(s_{i,j}\) shown in Figure 3. The three types of diagrams placed to the left of the identity diagrams are denoted as \(e_{i,j}^L\), \(p_i^L\), and \(s_{i,j}^L\). Similarly, those on the right are denoted as \(e_{i,j}^R\), \(p_i^R\), and \(s_{i,j}^R\). Figure \ref{Ogene} provides an example.

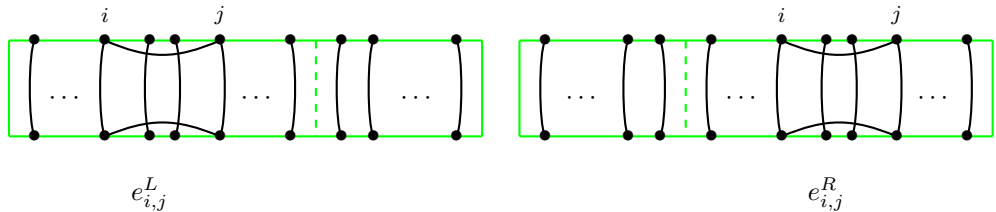
\begin{figure}[h]
	\begin{equation*}
		\begin{array}{c}
			\begin{tikzpicture}[scale=0.85] 
				\draw [green,thick](-0.4,1.5) --(7,1.5);
				\draw [green,thick](-0.4,0) --(7,0);
				\draw [green,thick](-0.4,1.5) --(-0.4,0);
				\draw [green,dashed,thick](4.4,1.5) --(4.4,0);
				\draw [green,thick](7,1.5) --(7,0);
			
				\node at (0,1.5) {\textbullet}; 
				\node at (1.1,1.5) {\textbullet};
				\node at (1.8,1.5) {\textbullet};
				\node at (2.2,1.5) {\textbullet};
				\node at (1.1,1.9) { \footnotesize$i$};
				\node at (2.9,1.9) { \footnotesize$j$};
				\node at (2.9,1.5) {\textbullet};
				\node at (4,1.5) {\textbullet};
				\node at (0,0) {\textbullet}; 
				\node at (1.1,0) {\textbullet};
				\node at (1.8,0) {\textbullet};
				\node at (2.2,0) {\textbullet};
				\node at (2.9,0) {\textbullet};
				\node at (4,0) {\textbullet};
				\node at (0.5,0.6) { $\cdots$};
				\node at (3.5,0.6) { $\cdots$};
				\draw [thick](0,1.5) .. controls (-0.1,1.3) and (-0.1,0.2) ..  (0,0); 
				\draw [thick](1.1,1.5) .. controls (1,1.3) and (1,0.2) ..  (1.1,0); 
				\draw [thick](2.2,1.5) .. controls (2.3,1.3) and (2.3,0.2) ..  (2.2,0); 
				\draw [thick](2.9,1.5) .. controls (3,1.3) and (3,0.2) ..  (2.9,0); 
				\draw [thick](4,1.5) .. controls (4.1,1.3) and (4.1,0.2) ..  (4,0); 
				\draw [thick](1.8,1.5) .. controls (1.7,1.3) and (1.7,0.2) ..  (1.8,0); 
				\node at (1.8,-0.9) { $e_{i,j}^L$};
				\draw [thick](1.1,1.5) .. controls (1.8,1.2) and (2.2,1.2) ..  (2.9,1.5);
				\draw [thick](1.1,0) .. controls (1.8,0.29) and (2.2,0.29) ..  (2.9,0);
				
				\node at (4.8,0) {\textbullet};
				\node at (4.8,1.5) {\textbullet};
				\draw [thick](4.8,1.5) .. controls (4.7,1.3) and (4.7,0.2) ..  (4.8,0); 
				
				\node at (5.3,0) {\textbullet};
				\node at (5.3,1.5) {\textbullet};
				\draw [thick](5.3,1.5) .. controls (5.2,1.3) and (5.2,0.2) ..  (5.3,0); 
				
				\node at (6.6,0) {\textbullet};
				\node at (6.6,1.5) {\textbullet};
				\draw [thick](6.6,1.5) .. controls (6.7,1.3) and (6.7,0.2) ..  (6.6,0); 
				\node at (6,0.6) { $\cdots$};
			\end{tikzpicture}
		\end{array}~~~~~~~~
		\begin{array}{c}
			\begin{tikzpicture}[scale=0.85] 
				\draw [green,thick](-3,1.5) --(4.4,1.5);
				\draw [green,thick](-3,0) --(4.4,0);
				\draw [green,dashed,thick](-0.4,1.5) --(-0.4,0);
				\draw [green,thick](-3,1.5) --(-3,0);
				\draw [green,thick](4.4,1.5) --(4.4,0);
			
				\node at (0,1.5) {\textbullet}; 
				\node at (1.1,1.5) {\textbullet};
				\node at (1.8,1.5) {\textbullet};
				\node at (2.2,1.5) {\textbullet};
				\node at (1.1,1.9) { \footnotesize$i$};
				\node at (2.9,1.9) { \footnotesize$j$};
				\node at (2.9,1.5) {\textbullet};
				\node at (4,1.5) {\textbullet};
				\node at (0,0) {\textbullet}; 
				\node at (1.1,0) {\textbullet};
				\node at (1.8,0) {\textbullet};
				\node at (2.2,0) {\textbullet};
				\node at (2.9,0) {\textbullet};
				\node at (4,0) {\textbullet};
				\node at (0.5,0.6) { $\cdots$};
				\node at (3.5,0.6) { $\cdots$};
				\draw[thick] (0,1.5) .. controls (-0.1,1.3) and (-0.1,0.2) ..  (0,0); 
				\draw [thick](1.1,1.5) .. controls (1,1.3) and (1,0.2) ..  (1.1,0); 
				\draw [thick](2.2,1.5) .. controls (2.3,1.3) and (2.3,0.2) ..  (2.2,0); 
				\draw [thick](2.9,1.5) .. controls (3,1.3) and (3,0.2) ..  (2.9,0); 
				\draw [thick](4,1.5) .. controls (4.1,1.3) and (4.1,0.2) ..  (4,0); 
				\draw [thick](1.8,1.5) .. controls (1.7,1.3) and (1.7,0.2) ..  (1.8,0); 
				
				\draw [thick](1.1,1.5) .. controls (1.8,1.2) and (2.2,1.2) ..  (2.9,1.5);
				\draw [thick](1.1,0) .. controls (1.8,0.29) and (2.2,0.29) ..  (2.9,0);
				
				\node at (1.8,-0.9) { $e_{i,j}^R$};
				
				\node at (-0.8,0) {\textbullet};
				\node at (-0.8,1.5) {\textbullet};
				\draw [thick](-0.8,1.5) .. controls (-0.7,1.3) and (-0.7,0.2) ..  (-0.8,0); 
				
				\node at (-1.3,0) {\textbullet};
				\node at (-1.3,1.5) {\textbullet};
				\draw [thick](-1.3,1.5) .. controls (-1.4,1.3) and (-1.4,0.2) ..  (-1.3,0); 
				
				\node at (-2.6,0) {\textbullet};
				\node at (-2.6,1.5) {\textbullet};
				\draw [thick](-2.6,1.5) .. controls (-2.7,1.3) and (-2.7,0.2) ..  (-2.6,0); 
				\node at (-2,0.6) { $\cdots$};
			\end{tikzpicture}
		\end{array}
	\end{equation*}

	\vskip -0.3cm
	\caption{An example of generators of $P_m \otimes P_n$}
	\label{Ogene}
\end{figure}

Let $d$ be a generator in $P_m \otimes P_n$, $v$ be a walled half-diagram with the index $(U~;~T, L, R)$, and $T+L+R=r$ be the total number of labeled blocks. The resulting diagram after concatenation of $d$ and $v$ is denoted as $v',$ and its index is $(U'~;~T', L',R')$. All changes to these indices are listed in Lemma \ref{order}(see also the proof of Theorem 4.3 in \cite{BVC} for $(n,r)$-diagrams). These changes come directly from how the generators act.   For example, the first case in Lemma \ref{order} happens when the generator is \(d = p_i^L\) (or \(p_i^R\)), where 
$i$ belongs to a block \(B\) that crosses the wall, and \(i\) is the only node of block \(B\) on the left (respectively, right) side of the wall.  

\begin{lemma} \label{order}
	
There are three cases in which the index changes but the number $r$ of labeled blocks remains constant.

   (I) $(U'~;~ T', L', R')=(U~;~T-1, L, R+1 )$ (or $=(U~;~T-1, L +1, R )$)

(II) $(U'~;~ T', L', R')=(U -1~;~ T+1,L-1,R)$  (or=$(U -1~;~ T+1, L, R-1)$)

   (III) $(U'~;~ T', L', R')=(U -1~;~ T, L, R) $

There are two cases in which the index changes and the total number $r$ of labeled blocks decreases, as shown below. 

   (IV) $(U'~;~ T',L',R')=(U~ ;~ T, L-1, R)$ (or $=(U ~ ;~  T,L,R-1)$)

  (V) $(U'~;~ T',L',R')=(U ~;~ T-1, L ,R )$

\end{lemma}

Let  $I_{m | n}^r$  be the index set  consisting of all possible  indices  $(U~;~T,L,R)$ of the $(m | n, r)$-walled half-diagrams.  
We adopt  a lexicographic order on the index set \( I_{m|n}^r \): For two indices \((U~;~ T, L, R)\) and \((U'~;~ T', L', R')\), we write \((U~;~ T, L, R) < (U'~;~ T', L', R')\) if and only if the former is smaller when compared sequentially by the priorities \( U, T, L, R \).

This order is compatible with Lemma \ref{order}, as the action of any generator $d \in P_m \otimes P_n$ strictly decreases the index lexicographically.
 For example: when the generator decreases \( U \) (as in cases (II) or (III)), decreases \( T \) with \( U \) unchanged (case (I)), decreases \( L\) or \(R \) with \( U, T \) fixed (case (IV)), or solely reduces \( T \) (case (V)), the order is strictly lowered.

When we consider restricting the \(P_{m+n}\) module \(V_{m+n}(r)\) to the subalgebra \(P_{m} \otimes P_{n}\), we denote this as \(V_{m+n}(r) \downarrow_{m,n}^{m+n}\). It admits a filtration via lexicographic order. Each  \textit{walled module} is a submodule in this filtration, defined as follows:

\begin{defin}\label{wall}  
	Let $(U~;~T, L, R)$ be an index for $(m | n, r)$-walled half-diagrams. The $(m|n)$-walled module $W_{m | n}(U~;~T, L, R)$ is the submodule of $V_{m+n}(r) \downarrow_{m,n}^{m+n}$ spanned by all $(m|n, r)$-walled half-diagrams with indices in $I_{m|n}^r$ lexicographically $\leq (U~;~T, L, R)$.  
\end{defin}

This provides a filtration:  
\[
0 \subset \cdots \subset W_{m|n}(U'; T', L', R') \subset W_{m|n}(U; T, L, R) \subset \cdots \subset V_{m+n}(r)\!\!\downarrow_{m,n}^{m+n}. \eqno(*)
\]

	Let $V_{m|n}(U~;~T, L, R):=W_{m|n}(U~;~T, L, R)/ W_{m|n}(U^\prime~;~T^\prime, L^\prime, R^\prime)$ be a subquotient of the chain ($*$). Moreover, this subquotient is naturally isomorphic to the vector space spanned by all $(m|n, r)$-walled half-diagrams of index $(U~;~T, L, R)$.
Note that if a generator $d \in P_m \otimes P_n$ acts so that the index changes as in cases (IV) or (V) of Lemma \ref{order}, the action is zero, because it lowers the total number $r$ of labeled blocks.

\subsection{Diagram Decomposition via Inflation}

By decomposing walled half-diagrams, we can characterize the subquotients in filtration (\(*\)). As detailed in Proposition \ref{iso}, a subquotient can be decomposed into the tensor product of two smaller half-diagram modules, corresponding to `Part I', and the induction of symmetric groups, corresponding to `Part II'.
Based on Theorem \ref{par-mul} by Bowman et al., the restriction $V_{m+n}(r)\!\!\downarrow_{m,n}^{m+n}$ admits a filtration by standard modules. In fact, `Part II' can be further decomposed into Specht modules, which then combine with `Part I' to form standard modules. Subsequently, we shall utilize the decomposition of `Part II' to demonstrate that each index contributes with multiplicity one, while `Part I' corresponds to the Diophantine equation (E1).

\begin{prop}\label{iso}
As a $P_m \otimes P_n$-module,
 $V_{m|n}(U ~;~ T,L,R )$ is isomorphic to
	$$	
\underbrace{\big (~V(m,U+T+L)~	\boxtimes~V(n, U+T+R) ~\big ) }_{\text{Part I}}
~\otimes~
\underbrace{\big( \mathbb{C}\mathfrak{S}_{U} \boxtimes \mathbb{C}\mathfrak{S}_{T}\big)}_{\text{Part II}},$$
where the tensor product $\otimes$ over algebra  $\mathbb{C}\mathfrak{S}_{U+T+L,U+T+R}$. 
\end{prop}

\proof

We'll first consider some specific cases and then deal with  the general case.

\smallskip

\textbf{Case:}   $\mathbf{(0~;~0, L, R)}$. 

Given a walled diagram \(v\) with index \( (0~;~0, L, R) \), since it has no strings crossing the wall, it is essentially a juxtaposition of an \( (m, L) \)-half-diagram \( v_1 \) and an \( (n, R) \)-half-diagram \( v_2 \). Therefore, we have the following isomorphism: 
\begin{eqnarray*}
V_{m|n}(0 ~;~ 0, L, R ) &\cong&  V_{m}(L)~ \boxtimes~ V_{n}(R)\\
	& \cong &\Big(V(m,L)\otimes_{\mathbb{C}\mathfrak{S}_L}\mathbbm{1}_L\Big) \boxtimes \Big(V(n,R)\otimes_{\mathbb{C}\mathfrak{S}_R}\mathbbm{1}_R\Big) \\
	& \cong& \Big(V(m,L)\boxtimes V(n,R)\Big) \otimes_{\mathbb{C}\mathfrak{S}_{L, R}} \Big(\mathbbm{1}_L\boxtimes\mathbbm{1}_R\Big) .
\end{eqnarray*}

\textbf{Case:}  $\mathbf{(U~;~T, 0, 0)}$.

Let \(v\) be an \((m | n)\)-walled half-diagram with index \((U~;~T, 0, 0)\). We will proceed with its decomposition, as an example illustrated in Figure \ref{TT00}.   First, the part to the right of the wall is rotated clockwise to the bottom, resulting in a partition diagram with \(m\) dots on top and \(n\) dots below. Note that the strings originally passing through the wall become propagating strings in the new diagram.  This diagram is then standardized (see Section 1.1), meaning that only a single string connects the blocks at the top and bottom. Additionally, the leftmost dots of each block are used for connections.

As shown in Figure~\ref{TT00}(1), we perform ``surgery'' on the diagram as follows:  

First, cut the propagating strings with two horizontal lines, dividing the diagram into three parts: the upper part \(v_1\), the middle part \(\sigma\), and the lower part \(v_2\). Here, \(v_1\) and \(v_2\) belong to \(V(m, U+T)\) and \(V(n, U+T)\), respectively.  

The middle part consists of all propagating strings determining a permutation \(\sigma \in \mathfrak{S}_{U+T}\). This permutation can be decomposed further as:  
\[
\alpha(\sigma_1 \otimes \sigma_2)\alpha^{\prime} \in~ _{\mathbb{C}\mathfrak{S}_{U+T}}\!\left( \mathbb{C}\mathfrak{S}_{U} \boxtimes \mathbb{C}\mathfrak{S}_{T} \right)_{\mathbb{C}\mathfrak{S}_{U+T}},
\]  
where \(\sigma_1 \in \mathfrak{S}_{U}\) and \(\sigma_2 \in \mathfrak{S}_{T}\) are determined by the labeled and non-labeled blocks, respectively.  

Since the tensor product is balanced over \(\mathbb{C}\mathfrak{S}_{U+T,U+T}\), both \(v_1 \cdot \alpha\) and \(v_2 \cdot \alpha^{\prime}\) remain in \(V(m, U+T)\) and \(V(n, U+T)\), respectively.  

Therefore, we obtain the isomorphism:  
\[
V_{m | n}(U~;~ T, 0, 0 ) \cong \left(V(m, U+T) \boxtimes V(n, U+T)\right) \otimes \left( \mathbb{C}\mathfrak{S}_{U} \boxtimes \mathbb{C}\mathfrak{S}_{T} \right).
\]  

As shown in Figure~\ref{TT00}(2), this case also corresponds to a half-diagram: when cutting the propagating strings, all strings degenerate into labeled dots, yielding a half-diagram in \(V_m(U+T) \boxtimes V_n(U+T)\).

\begin{figure}[ht]
			\begin{tikzpicture}
				
				%
				
				\draw [green,thick](-1.9,0.7) --(6.2,0.7);
				\draw [green,thick](2.4,1.5) --(2.4,0);
				
				\node at (-1.5,0.7) {\textbullet}; 
				\node at (-1,0.68) {{\huge\textopenbullet}}; 
				\node at (-0.5,0.7) {\textbullet}; 
				\node at (0,0.68) {{\huge\textopenbullet}}; 
				\node at (0.5,0.68) {{\huge\textopenbullet}};
				\node at (1,0.7) {\textbullet};
				\node at (1.5,0.7) {\textbullet};
				\node at (2,0.7) {\textbullet};
				\node at (-1.5,1) {\small$1$}; 
				\node at (-1,1) {\small$2$}; 
				\node at (-0.5,1) {\small$3$}; 
				\node at (0,1) {\small$4$}; 
				\node at (0.5,1) {\small$5$};
				\node at (1,1) {\small$6$};
				\node at (1.5,1) {\small$7$};
				\node at (2,1) {\small$8$};

				\node at (2.8,0.7) {\textbullet};
				\node at (3.3,0.68) {{\huge\textopenbullet}};
				\node at (3.8,0.68) {{\huge\textopenbullet}};
				\node at (4.3,0.7) {\textbullet};
				\node at (4.8,0.7) {\textbullet};
				\node at (5.3,0.7) {\textbullet};
				\node at (5.8,0.7) {\textbullet};
				
				\node at (2.8,1) {\small$7'$};
				\node at (3.3,1) {\small$6'$};
				\node at (3.8,1) {\small$5'$};
				\node at (4.3,1) {\small$4'$};
				\node at (4.8,1) {\small$3'$};
				\node at (5.3,1) {\small$2'$};
				\node at (5.8,1) {\small$1'$};
				\draw [thick] (-1.5,0.7)  .. controls (-1.2,0.4) and (-0.8,0.4) ..  (-0.5,0.7) ; 
				\draw [thick](-1,0.7) .. controls (-0.7,0.45) and (-0.3,0.45) ..  (0,0.7) ; 
				\draw[red,thick] (0,0.7) .. controls (0.3,0.2) and (3,0.2) ..  (3.3,0.7) ; 
				\draw[red,thick]  (0.5,0.7)   .. controls (0.8,0) and (3.5,0) ..  (3.8,0.7) ; 
				\draw [blue,thick] (1.5,0.7)  .. controls (1.8,0.5) and (2.5,0.5) ..  (2.8,0.7) ; 
				\draw [blue,thick] (2,0.7)  .. controls (2.3,0.2) and (4,0.2) ..  (4.3,0.7) ; 
				\draw [thick](4.8,0.7)  .. controls (5,0.4) and (5.6,0.4) ..  (5.8,0.7) ; 
				
				\draw[->,thick] (3.4,-0.2) arc(0:-120:1.5);
			\end{tikzpicture}

			\begin{tikzpicture}
				
				\draw [green,thick](-1.9,0.7) --(2.4,0.7);
				
				\draw [green,thick](-1.9,-1) --(1.9,-1);
				\draw [green,dashed,thick](-1.9,0.2) --(2.3,0.2);
				\draw [green,dashed,thick](-1.9,-0.5) --(2.3,-0.5);			
				
							\node at (-3,-0.15) {\small$ 	\longleftrightarrow $};	
				\node at (-2.2,0.5) {\small$  V(8, 4)  \ni$};	
				\node at (-2.2,-0.8) {\small$   V(7, 4)\ni$};	
				\node at (-1.5,0.7) {\textbullet}; 
				\node at (-1,0.68) {{\huge\textopenbullet}}; 
				\node at (-0.5,0.7) {\textbullet}; 
				\node at (0,0.68) {{\huge\textopenbullet}}; 
				\node at (0.5,0.68) {{\huge\textopenbullet}};
				\node at (1,0.7) {\textbullet};
				\node at (1.5,0.7) {\textbullet};
				\node at (2,0.7) {\textbullet};
				\node at (-1.5,1) {\small$1$}; 
				\node at (-1,1) {\small$2$}; 
				\node at (-0.5,1) {\small$3$}; 
				\node at (0,1) {\small$4$}; 
				\node at (0.5,1) {\small$5$};
				\node at (1,1) {\small$6$};
				\node at (1.5,1) {\small$7$};
				\node at (2,1) {\small$8$};
				
				\draw [thick](-1.5,0.7)  .. controls (-1.2,0.4) and (-0.8,0.4) ..  (-0.5,0.7) ; 
				\draw [thick](-1,0.7) .. controls (-0.7,0.45) and (-0.3,0.45) ..  (0,0.7) ;

				\node at (-1.5,-1.01) {\textbullet}; 
				\node at (-1,-1.01) {\textbullet}; 
				\node at (-0.5,-1.01) {\textbullet}; 
				\node at (0,-1.01) {\textbullet};

				\node at (-1.5,-1.3) {$1'$}; 
				\node at (-1,-1.3) {$2'$}; 
				\node at (-0.5,-1.3) {$3'$}; 
				\node at (0,-1.3) {$4'$}; 
				\node at (0.5,-1.3) {$5'$};
				\node at (1,-1.3) {$6'$};
				\node at (1.5,-1.3) {$7'$};
				
				\draw[thick] (-1.5,-1)  .. controls (-1.2,-0.7) and (-0.8,-0.7) ..  (-0.5,-1) ; 
				%

				\draw[red,thick] (-1,0.7)   .. controls (-1,-0.2) and (1,0) ..  (1,-1) ; 
				\draw [red,thick](0.5,0.7)  .. controls (0.7,-0.2) and (0.3,0) ..  (0.5,-1) ; 
				
				\draw [blue,thick](1.5,0.7)   .. controls (1.6,-0.2) and (1.3,0) ..  (1.5,-1) ; 
				\draw [blue,thick](2,0.7)   .. controls (2,-0.4) and (-0.1,-1) ..  (0,-1) ; 			
				
				\node at (4,0.7) {$\tiny	\begin{pmatrix}
						1 & 2 & 3 & 4 \\
						1 & 2 &3&4 \\
					\end{pmatrix}$};	
				
				\node at (4,-0.1) {$\tiny	\begin{pmatrix}
						1 & 2  \\
						2 & 1  \\
					\end{pmatrix} \o 	\begin{pmatrix}
						3 & 4  \\
						4 & 3 \\
					\end{pmatrix} $};	
				
				\node at (4,-1) {$\tiny	\begin{pmatrix}
						1 &2 & 3 & 4 \\
						2 & 3 &1&4 \\
					\end{pmatrix}$};	
				
				\node at (7.5,-0.1) {$  \in~ _{\mathbb{C}\mathfrak{S}_{4}}\big( \mathbb{C}\mathfrak{S}_{2} \boxtimes \mathbb{C}\mathfrak{S}_{2} \big)_{\mathbb{C}\mathfrak{S}_{4}}$};

				\node at (0.5,-1.07) {{\huge\textopenbullet}};
				\node at (1,-1.07) {{\huge\textopenbullet}};
				\node at (1.5,-1.01) {\textbullet};

				\node at (-0.74,0.18) {{\huge\textopenbullet}}; 
				\node at (0.76,-0.53) {{\huge\textopenbullet}};

				\node at (0.55,0.18) {{\huge\textopenbullet}};
				\node at (0.45,-0.53) {{\huge\textopenbullet}};

				\node at (1.5,0.7) {\textbullet};
				\node at (1.52,0.2) {\textbullet};
				\node at (1.45,-0.5) {\textbullet};
				
				\node at (2,0.7) {\textbullet};
				\node at (1.88,0.2) {\textbullet};
				\node at (1.1,-0.5) {\textbullet};
				
				\node at (-0.74,0.18) {{\huge\textopenbullet}}; 
				\node at (0.76,-0.53) {{\huge\textopenbullet}};

				\node at (0.55,0.18) {{\huge\textopenbullet}};
				\node at (0.45,-0.53) {{\huge\textopenbullet}}; 
				\node at (1.5,-1.9) {(1)~Decomposition}; 
			\end{tikzpicture}
		\begin{tikzpicture}
			
			\draw [green,thick](-1.9,0.7) --(2.1,0.7);
			\draw [green,thick](2.7,0.7) --(6.2,0.7);
			\draw [green,dashed,thick](2.4,1.5) --(2.4,0);
			
			\node at (8,0.7) {$\in V_{8}(4) \boxtimes V_{7}(4)$}; 
			\node at (2.7,-0.17) {(2) ~Corresponding half-diagram}; 
			
			\node at (-1.5,0.68) {\textbullet}; 
			\node at (-1,0.65) {{\huge\textopenbullet}}; 
			\node at (-0.5,0.68) {\textbullet}; 
			\node at (0,0.65) {{\huge\textopenbullet}}; 
			\node at (0.5,0.65) {{\huge\textopenbullet}};
			\node at (1,0.68) {\textbullet};
			\node at (1.5,0.65) {{\huge\textopenbullet}};
			\node at (2,0.65) {{\huge\textopenbullet}};
			\node at (-1.5,1) {\small$1$}; 
			\node at (-1,1) {\small$2$}; 
			\node at (-0.5,1) {\small$3$}; 
			\node at (0,1) {\small$4$}; 
			\node at (0.5,1) {\small$5$};
			\node at (1,1) {\small$6$};
			\node at (1.5,1) {\small$7$};
			\node at (2,1) {\small$8$};

			\node at (2.8,0.65) {{\huge\textopenbullet}};
			\node at (3.3,0.65) {{\huge\textopenbullet}};
			\node at (3.8,0.65) {{\huge\textopenbullet}};
			\node at (4.3,0.65) {{\huge\textopenbullet}};
			\node at (4.8,0.68) {\textbullet};
			\node at (5.3,0.68) {\textbullet};
			\node at (5.8,0.68) {\textbullet};
			
			\node at (2.8,1) {\small$7'$};
			\node at (3.3,1) {\small$6'$};
			\node at (3.8,1) {\small$5'$};
			\node at (4.3,1) {\small$4'$};
			\node at (4.8,1) {\small$3'$};
			\node at (5.3,1) {\small$2'$};
			\node at (5.8,1) {\small$1'$};
			\draw [thick](-1.5,0.7)  .. controls (-1.2,0.4) and (-0.8,0.4) ..  (-0.5,0.7) ; 
			\draw [thick](-1,0.7) .. controls (-0.7,0.45) and (-0.3,0.45) ..  (0,0.7) ; 
			
			\draw [thick](4.8,0.7)  .. controls (5,0.4) and (5.6,0.4) ..  (5.8,0.7) ;

		\end{tikzpicture} 
		\vskip -0.3cm
		\caption{An example for the case  $(U~;~T, 0, 0)$ }
		\label{TT00}
\end{figure}
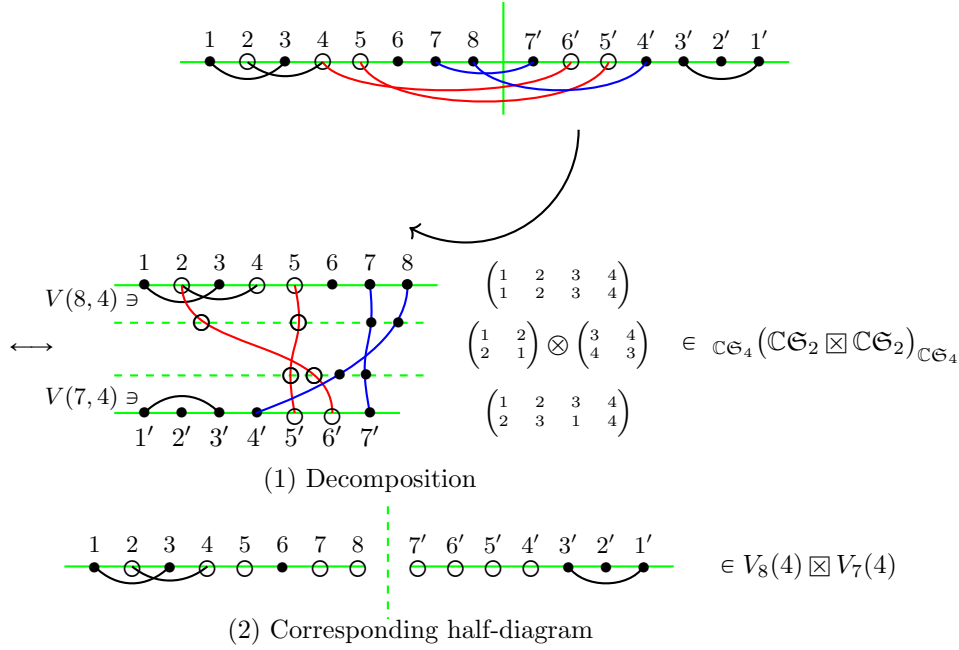

\textbf{General case.} $\mathbf{(U~;~T, L, R)}$.

Note that once the propagating strings are cut and degenerate into labeled dots, the resulting walled diagram will have \( U + T+ L \) and \( U + T+ R \) labeled blocks on the left and right sides, respectively. In comparison to Case $(U~;~T, 0, 0)$, \(\big( \mathbb{C}\mathfrak{S}_{U} \boxtimes \mathbb{C}\mathfrak{S}_{T} \big)\) is a  $\mathbb{C}\mathfrak{S}_{U+T+L,U+T+R}$-module.
 $\qed$

 \subsection{Multiplicity Reduction}
 
 The multiplicity \( E_{p,q}^r \) counts how many times \( \Delta_m((p)) \boxtimes \Delta_n((q)) \) appears in the restriction \( \Delta_{m+n}((r))\!\!\downarrow_{P_m \otimes P_n} \). By Schur's lemma and the semisimplicity of \( P_m \otimes P_n \):
 \[
 E_{p,q}^r = \dim \Hom_{P_m \otimes P_n}\left( V_m(p) \boxtimes V_n(q),\ V_{m+n}(r)\!\!\downarrow_{m,n}^{m+n} \right).
 \]
 
 The filtration (*) decomposes the restriction module into subquotients \( V_{m|n}(U; T, L, R) \).
 
 Note that $$ V_{m}(p) \boxtimes V_{n}(q) \cong (V(m, p)\otimes \mathbbm{1})\boxtimes (V(n, q)\otimes \mathbbm{1}) \cong \underbrace{ (V(m, p) \boxtimes V(n, q))}_{\text{Part I'}}\otimes\underbrace{ (\mathbbm{1} \boxtimes \mathbbm{1})}_{\text{Part II'}} $$.
 
 For a nonzero homomorphism \( \Delta_m((p)) \boxtimes \Delta_n((q)) \to V_{m|n}(U; T, L, R) \), the Part I components must satisfy:
 \[
 \underbrace{V(m,p) \boxtimes V(n,q)}_{\text{Part I}'} \cong \underbrace{V(m, U+T+L) \boxtimes V(n, U+T+R)}_{\text{Part I}}.
 \]
 This forces:
 \[
 p = U + T + L \quad \text{and} \quad q = U + T + R.
 \]
 Combined with the original condition \( T + L + R = r \), this corresponds to the equations (E1).
 
 Furthermore, by Proposition \ref{aalc}, the multiplicity reduces to computations involving Part II. The Hom-space dimension reduces to detecting trivial modules in the induced modules:
 \[
 \begin{aligned}
 	&\dim \Hom\left(\mathbbm{1}_{U+T+L} \boxtimes \mathbbm{1}_{U+T+R},~ \mathbb{C}\mathfrak{S}_U \boxtimes \mathbb{C}\mathfrak{S}_T\right) \\
 \end{aligned}
 \]
 This dimension must equal 1. This follows because the group algebra decomposes as a bimodule$^1$\footnote{$^1$This decomposition arises from the Artin-Wedderburn theorem, which splits $\mathbb{C}\mathfrak{S}_n$ into matrix algebras $\bigoplus_{\lambda} \operatorname{End}(\Sp(\lambda))$, combined with the self-duality of irreducible representations ($\Sp(\lambda) \cong \Sp(\lambda)^*$) that implies $\operatorname{End}(S^\lambda) \cong S^\lambda \otimes S^\lambda$.}:
 \[
 \mathbb{C}\mathfrak{S}_n \cong \bigoplus_{\lambda \vdash n} \left(  \Sp(\lambda) \boxtimes  \Sp(\lambda) \right),
 \]
 where each index appears exactly once.
 
 For Case \( \mathbf{(U~;~T, 0, 0)} \):
 \begin{eqnarray*}
 	_{\mathbb{C}\mathfrak{S}_{U+T}}\big( \mathbb{C}\mathfrak{S}_{U} \boxtimes \mathbb{C}\mathfrak{S}_{T} \big)_{\mathbb{C}\mathfrak{S}_{U+T}}&\cong&  \bigoplus_{\lambda \vdash U} \left( \Sp(\lambda) \boxtimes \Sp(\lambda) \right)\boxtimes \bigoplus_{\mu \vdash T} \left( \Sp(\mu)\boxtimes \Sp(\mu) \right)\\
 	& \cong &\bigoplus_{\substack{\lambda \vdash U\\ \mu \vdash T}} \left( \Sp(\lambda) \boxtimes \Sp(\mu) \right)\boxtimes  \left( \Sp(\lambda) \boxtimes \Sp(\mu) \right).
 \end{eqnarray*}
Hence the dimension is reduced to the following, and by Lemma \ref{two-part} we obtain:
 \[
 \mathrm{dim}~ \mathrm{Hom}_{\mathbb{C}\mathfrak{S}_{U+T}}(\mathbbm{1}_{U+T},~\bigoplus_{\substack{\lambda \vdash U\\ \mu \vdash T}} \left( \Sp(\lambda) \boxtimes \Sp(\mu) \right))=c^{(U+T)}_{(U), (T)}=1.
 \]
 For the general case, since the restriction and induction of the trivial module remain trivial, it reduces to the Case $(U~;~T, 0, 0)$.
 
 In summary, the dimension of the Hom-space is determined by:
 \begin{eqnarray*}
 	\dim ~\Hom_{P_m \otimes P_n}\Big (V_{m}(p) \boxtimes V_{n}(q),~ V_{m | n}(U ~;~ T,L,R )\Big) &=&\left\{
 	\begin{array}{ll}
 		1 \vspace{0.2cm}&\text{if } p=U+T+L,\\
 		\vspace{0.2cm}&\text{and } q=U+T+R;\\
 		0 &\text{otherwise.}
 	\end{array}
 	\right. 
 \end{eqnarray*}
 Thus, each subquotient contributes 1 to the multiplicity if it satisfies the conditions in the theorem. Therefore, the multiplicity \( E_{p,q}^r \) equals the number of solutions to (E1).

\section{Geometric interpretation }

\subsection{Geometric problem} 
How many non-negative integer solutions exist for the following system of equations (E1)? 

First, by adding and subtracting the equations in (E1), we can easily derive that the parameters \( p \), \( q \), and \( r \) must satisfy the inequality \( |p - q| \leq r \leq p + q \).  
When $p$, $q$, and $r$ satisfy the strict triangle inequality $|p - q| < r < p + q$, and both $p$ and $q$ are greater than 0 (in which case $r$ must also be positive),
 consider triangle \(ABC\) with integer side lengths \(p\), \(q\), and \(r\), and let \(O\) be its incenter. When drawing a circle centered at \(O\) that intersects all three sides, the intersection points create symmetric segments on each side, as shown in Figure \ref{inters}(1). Specifically:   \(|AH_1| = |AH_2| = R;\)  \(|BG_1| = |BG_2| = L;\)   \(|CF_1| = |CF_2| = U;\)   \(|H_1F_2| = |F_1G_2| = |G_1H_2| = T.\)  
When the circle is tangent to all sides (i.e., the incircle), the parameter \( T = 0 \).

The constraints on the three side lengths are given by system (E1).
The problem is equivalent to  the following plane geometry question:
 the number of non-negative integer solutions to system (E1) corresponds precisely to the number of concentric circles centered at the triangle's incenter that intersect all three sides at integer lengths, as shown in Figure  \ref{inters}(2).

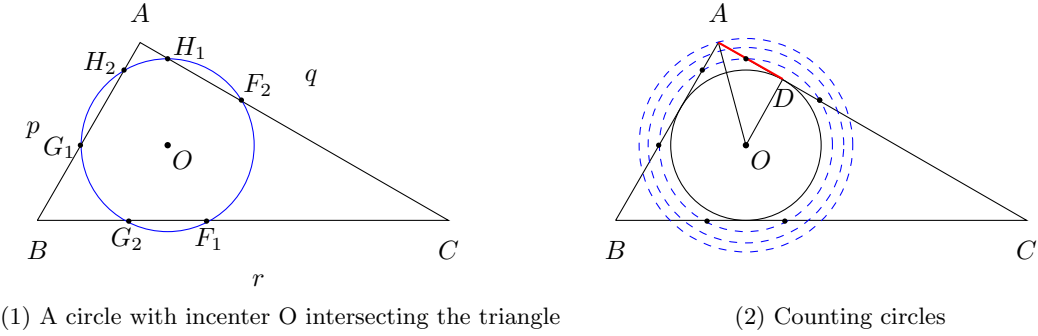
\begin{figure}[h]
	\begin{equation*}
		\begin{array}{c}
			\begin{tikzpicture}
				\coordinate (o) at (0,0);
				\coordinate (a) at (0.49,0.88) ;
				\coordinate (b) at ({(1-sqrt(3))/2}, {(1+sqrt(3))/2}) ;
				\coordinate (c) at (0:3);
				
				\draw ({-sqrt(3)},-1) -- ({(1-sqrt(3))/2}, {(1+sqrt(3))/2}) -- ({2+sqrt(3)},-1) -- ({-sqrt(3)},-1) ;
				\draw [blue] (0,0) circle [radius=1.15];
				\fill (0,0) circle (1.2pt);
					\node at (0.2, -0.2)  {$O$ };
				\coordinate[label=right:{$r$}] (r) at (1, -1.8 );
				\coordinate[label=right:{$p$}] (p) at (-2, 0.2 );
				\coordinate[label=right:{$q$}] (q) at (1.7, 0.9 );
				\fill (0,1.15) circle (1pt);
					\node at (0.3,1.30)  {$H_1$ };
				\fill (-0.58,1) circle (1pt);
								\node at (-0.9,1.1)  {$H_2$ };
				\fill (0.98,0.6) circle (1pt);
					\node at (1.2,0.8)  {$F_2$ };
				\fill (0.517,-1.01) circle (1pt);
					\node at (0.55,-1.25) {$F_1$ };
				\fill (-.517,-1.01) circle (1pt);
					\node at (-.54,-1.25) {$G_2$ };
				\fill (-1.16,0) circle (1pt);
								\node at (-1.44,-.03) {$G_1$ };
					\node at ({(1-sqrt(3))/2}, {(1+sqrt(3))/2+0.4})  {$A$ };
						\node at ({-sqrt(3)},-1.4) {$B$ };
					\node at ({2+sqrt(3)},-1.4)  {$C$ };
				\node at (1.5,-2.3) { \small \mbox{(1) A
						circle  with incenter} O  \mbox{intersecting the triangle}};
			\end{tikzpicture}
		\end{array}
		\begin{array}{c}
			\begin{tikzpicture}
				\coordinate (o) at (0,0);
				\coordinate (a) at (0.49,0.88) ;
				\coordinate (b) at ({(1-sqrt(3))/2}, {(1+sqrt(3))/2}) ;
				\coordinate (c) at (0:3);
					\node at (0.2, -0.2)  {$O$ };
					\node at ({(1-sqrt(3))/2}, {(1+sqrt(3))/2+0.4})  {$A$ };
					\node at ({-sqrt(3)},-1.4) {$B$ };
						\node at ({2+sqrt(3)},-1.4)  {$C$ };
				\draw ({-sqrt(3)},-1) -- ({(1-sqrt(3))/2}, {(1+sqrt(3))/2}) -- ({2+sqrt(3)},-1) -- ({-sqrt(3)},-1) ;
				\draw(0,0) circle [radius=1];
				\draw [dashed, blue] (0,0) circle [radius=1.15];
				\draw [dashed, blue] (0,0) circle [radius=1.3];
				\draw [dashed, blue] (0,0) circle [radius=1.42];
				\fill (0,0) circle (1.2pt);
					\draw (0, -0)--(0.49,0.88) ;
					\draw (0, -0)--({(1-sqrt(3))/2}, {(1+sqrt(3))/2}) ;
								\draw [red,thick](0.49,0.88) --({(1-sqrt(3))/2}, {(1+sqrt(3))/2}) ;
					\node at (0.5,0.63) {$D$ };
				\fill (0,1.15) circle (1pt);
				\fill (-0.58,1) circle (1pt);
				\fill (0.98,0.6) circle (1pt);
				\fill (0.98,0.6) circle (1pt);
				\fill (0.517,-1.01) circle (1pt);
				\fill (-.517,-1.01) circle (1pt);
				\fill (-1.16,0) circle (1pt);
				\node at (1.2,-2.3) { \small \mbox{ (2) Counting circles}};
			\end{tikzpicture}
		\end{array}
	\end{equation*}
	\caption{Circles intersecting line segments on three sides in integer lengths}
	\label{inters}
\end{figure}

By transforming the multiplicity problem into a geometric framework, 
we gain a more intuitive geometric understanding of the problem, as illustrated in Figure \ref{inters}(2). In this figure, $BC$ is the longest side of the triangle, and point $D$ denotes the incircle’s tangency point on side $AC$.

A circle with a radius that is either too large or too small cannot simultaneously intersect all three sides. Specifically, the lower bound for the radius is the inradius (incircle radius), while the upper bound is the distance $|AO|$ from vertex $A$ to the incenter $O$. This is equivalent to the requirement that the intersection point of the circle in question with side $AC$ must lie on segment $AD$ to ensure intersection with all three sides.

Consequently, the multiplicity is fundamentally determined by the length of the shortest segment formed by the incircle's tangency points along the sides—specifically, the length $|AD|$ of segment $AD$ in the figure.
Indeed, subsequent multiplicity calculations  demonstrate that the multiplicity equals the integer part of $|AD|$ plus one, i.e., $[|AD| ]  + 1$ (see Fig. \ref{Tri}).

\subsection{Lattice Point Enumeration} 
We now transform the problem of solving (E1) into a lattice point counting problem, as described in the following lemma.

\begin{lemma} \label{Lattice}
Let \( h = p + q - r \) and \( s = \min\{p, q\} \). Let $x$ and $y$ respectively replace $T$ and $U$. The system (E1) is equivalent to finding non-negative integer solutions to:  
\[
\left\{
\begin{array}{ll}
	y = -\dfrac{x}{2} + \dfrac{h}{2}, \\		
	y \leq -x + s.
\end{array}
\right.\eqno(\mathrm{E2})
\]  
\end{lemma}

\proof
 From the second and third equations of (E1):  
$L = p - T - U,$ $R = q - T - U.$
Substitute into the first equation:  
$T + (p - T - U) + (q - T - U) = r \implies p + q - 2U - T = r.$
Letting \( h = p + q - r \), we obtain $T + 2U = h.$
Let \( s = \min(p, q) \). Since \( L \geq 0 \) and \( R \geq 0 \), we have $T + U \leq s.$ Finally, setting \( x = T \) and \( y = U \), we arrive at the system (E2).  

Conversely, if \( x = T \) and \( y = U \) are non-negative integer solutions to (E2), then \( L = p - x - y \) and \( R = q - x - y \) are guaranteed to be non-negative (due to \( y \leq -x + s \)). Since \( x + 2y = h = p + q - r \), substituting back into the first equation of (E1) automatically satisfies the equality. $\qed$

\begin{figure}[h]
	\begin{equation*}
		\begin{tikzpicture}[scale=0.60]
			\shade[right color=gray!20,left color=gray!20]	(0,0) -- (0,6) -- (6,0) -- cycle;
			\draw[step=1cm,gray,very thin] (-1.5,-1.5) grid (10.5,6.9);
			\draw[thick,->] (0,0) -- (10,0) node[anchor=north west] {x};
			\draw[thick,->] (0,0) -- (0,6.5) node[anchor=south east] {y};
			\foreach \x in {0,1,2,3,4,5,6,7,8,9}
			\draw (\x cm,1pt) -- (\x cm,-1pt) node[anchor=north] {\footnotesize$\x$};
			\foreach \y in {0,1,2,3,4,5,6}
			\draw (1pt,\y cm) -- (-1pt,\y cm) node[anchor=east] {\footnotesize$\y$};
			
			\draw (0,6) -- (6,0);
			\draw[dashed,thick] (0,4) -- (8,0);
			\draw[dashed,thick]  (-1,3) -- (5,0);
			\filldraw[black] (8,0) circle (2pt);
			\filldraw[black] (6,1) circle (2pt);
			\filldraw[black] (4,2) circle (2pt);
			\draw[red] (4,2) circle (4pt) ;
			\filldraw[black] (2,3) circle (2pt);
			\draw[red] (2,3) circle (4pt) ;
			\filldraw[black] (0,4) circle (2pt);
			\draw[red] (0,4) circle (4pt) ;
			\filldraw[black] (5,0) circle (2pt);
			\draw[red] (5,0) circle (4pt) ;
			\filldraw[black] (3,1) circle (2pt);
			\draw[red] (3,1) circle (4pt) ;
			\filldraw[black] (1,2) circle (2pt);
			\draw[red] (1,2) circle (4pt) ;
							\node at (5,-1) {{\blue $ \parallel$}};
			\node at (5,-1.7) {{\blue \Large$h$}};
						\node at (8,-1) {{\blue $ \parallel$}};
			\node at (8,-1.7) {{\blue  \Large$h$}};
				\node at (6,-1) {{\blue $ \parallel$}};
			\node at (6,-1.7) {{\blue \LARGE$s$}};
				\draw [red,<-](3,0) -- (5,0);
				\draw [red,->](3,0) -- (3,0.8);
				\draw [red,<-](1,1) -- (3,1);
			\draw [red,->](1,1) -- (1,1.8);
		\end{tikzpicture}
	\end{equation*}
	\caption{Number of lattice points }
	\label{latti}
\end{figure}
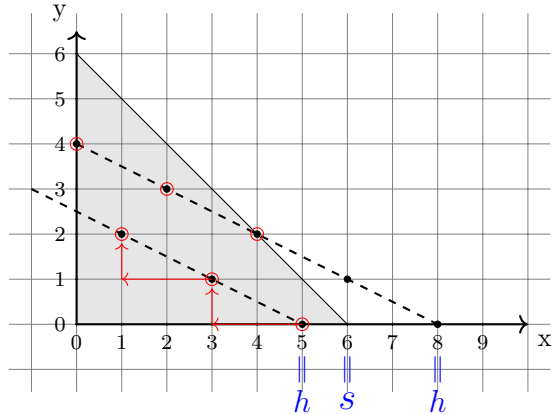

As shown in Figure \ref{latti}, $E_{p,q}^{r}$ counts the lattice points lying on the line $y = -\frac{1}{2}x + \frac{h}{2}$ that fall within the shaded triangular region. These points are enumerated through an L-shaped traversal: starting from $(h, 0)$, we iteratively move two units left followed by one unit up to reach subsequent lattice points along the line, counting only those points that remain within the shaded region. Analyzing the relationship between $s$ and $h$ yields:

\begin{eqnarray*}
	E_{p, q}^{r}&=&\left\{
	\begin{array}{ll}
		0 \vspace{0.2cm}&\mbox{if~~$2s< h$~ or ~$h< 0$,}\\\vspace{0.2cm}
		\left[ \frac{2s-h}{2} \right]+1&\mbox{if~~$ h \leqslant 2s \leqslant 2h$,}\\
		\left[ \frac{h}{2} \right]+1 &\mbox{if~~$ 0< h\leqslant s$.}
	\end{array}
	\right. 
\end{eqnarray*}

Figure \ref{latti} illustrates the cases $s=6$ with $h=8$ and $h=5$, corresponding respectively to the second and third conditions in the solution.
After substituting the parameters \(p\), \(q\), and \(r\),  the final expression is obtined:

\begin{thm}\label{Expression} Keep the notations as above.  We have
	\begin{equation*}
E_{p, q}^{r}=\left\{
	\begin{array}{ll}
		0 \vspace{0.3cm}&\mbox{if~~$p+q< r$~ or ~$r< | p-q |$,}\\\vspace{0.3cm}
		\left[ \frac{r-|p-q|}{2} \right]+1&\mbox{if~~$ | p-q |\leqslant r \leqslant \frac{p+q+|p-q|}{2} $,}\\
		\left[ \frac{p+q-r}{2} \right]+1 &\mbox{if~~$ \frac{p+q+|p-q|}{2}< r \leqslant p+q$.}
	\end{array}
	\right.
\end{equation*} 
\end{thm}
\vspace{0.3cm}

Upon observing the conditions, we conclude: When $p$, $q$, and $r$ satisfy the triangular inequality, a triangle with these three numbers as side lengths can be formed. Moreover, the structure constant $E_{p, q}^{r}$  is determined by the tangent distance of opposite vertex of the largest side of the triangle relative to the incircle. Figure \ref{Tri} illustrates that the tangent distance is $(r-| p-q |)/2$ when $r$ is not the length of the largest side (left diagram, and when $ | p-q |\leqslant r \leqslant \frac{p+q+|p-q|}{2} $), and $(p+q-r)/2$ when $r$ is the length of the largest side (right diagram, and when \( \frac{p+q+|p-q|}{2}< r \leqslant p+q\)). 

\begin{figure}[h]
	\begin{equation*}
		\begin{array}{c}
				\begin{tikzpicture}
				\coordinate (o) at (0,0);
				\coordinate (a) at (0.49,0.88) ;
				\coordinate (b) at ({(1-sqrt(3))/2}, {(1+sqrt(3))/2}) ;
				\coordinate (c) at (0:3);
				
				\draw ({-sqrt(3)},-1) -- ({(1-sqrt(3))/2}, {(1+sqrt(3))/2}) -- ({2+sqrt(3)},-1) -- ({-sqrt(3)},-1) ;
				\draw(0,0) circle [radius=1];
				\fill (0,0) circle (1.2pt);
				\coordinate[label=right:{$p$}] (p) at (1, -1.8 );
				\coordinate[label=right:{$r$}] (r) at (-2, 0.2 );
				\coordinate[label=right:{$q$}] (q) at (1.7, 0.9 );
				
				\draw[decorate,decoration={brace,raise=4pt,amplitude=0.1cm},red!100] (b) -- (a)node[black,midway,xshift=0.3cm,yshift=0.66cm]{$\frac{r-|p-q|}{2} $};
				\node at (1.5,-2.3) {$r$ \small \mbox{not the largest side}};
			\end{tikzpicture}
		
		\end{array}
		\begin{array}{c}
			\begin{tikzpicture}

			\coordinate (o) at (0,0);
			\coordinate (a) at (0.49,0.88) ;
			\coordinate (b) at ({(1-sqrt(3))/2}, {(1+sqrt(3))/2}) ;
			\coordinate (c) at (0:3);
			
			\draw ({-sqrt(3)},-1) -- ({(1-sqrt(3))/2}, {(1+sqrt(3))/2}) -- ({2+sqrt(3)},-1) -- ({-sqrt(3)},-1) ;
			\draw(0,0) circle [radius=1];
			\fill (0,0) circle (1.2pt);
			\coordinate[label=right:{$r$}] (r) at (1, -1.8 );
			\coordinate[label=right:{$p$}] (p) at (-2, 0.2 );
			\coordinate[label=right:{$q$}] (q) at (1.7, 0.9 );
			\draw[decorate,decoration={brace,raise=4pt,amplitude=0.1cm},red!100] (b) -- (a)node[black,midway,xshift=0.3cm,yshift=0.66cm]{$\frac{p+q-r}{2}$};
			\node at (1.5,-2.3) {$r$ \small \mbox{the largest side}};
		\end{tikzpicture}
		\end{array}
	\end{equation*}
	\caption{Triangle interpretation }
	\label{Tri}
\end{figure}
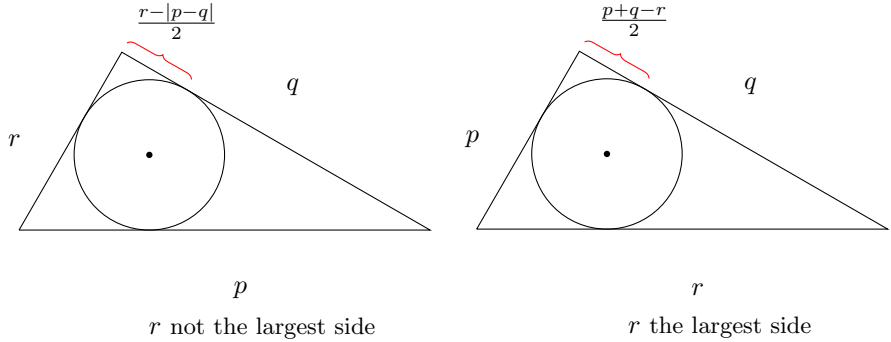

\subsection{Conic section } 
When both \( p \) and \( q \) are greater than 0 and satisfy the triangle inequality \( |p - q| < r < p + q \), the formula can also be interpreted through conic sections. As illustrated in Figure \ref{conic}(1), when the side of length \(r\) is the largest, an ellipse is determined: the length \(r\) is considered as the focal distance, and the sum of distances from any point on the ellipse to the two foci is kept constant at \(p+q\) (i.e., the major axis).

Similarly, as shown in Figure \ref{conic}(2), when the side of length \(r\) is not the largest, a hyperbola is determined: the length \(r\) serves as the focal distance, and the constant difference in distances from any point on the hyperbola to the two foci is \(|p-q|\) (i.e., the distance between two vertices).

\begin{figure}[ht]
	\centering
	\includegraphics[width=0.9\textwidth]{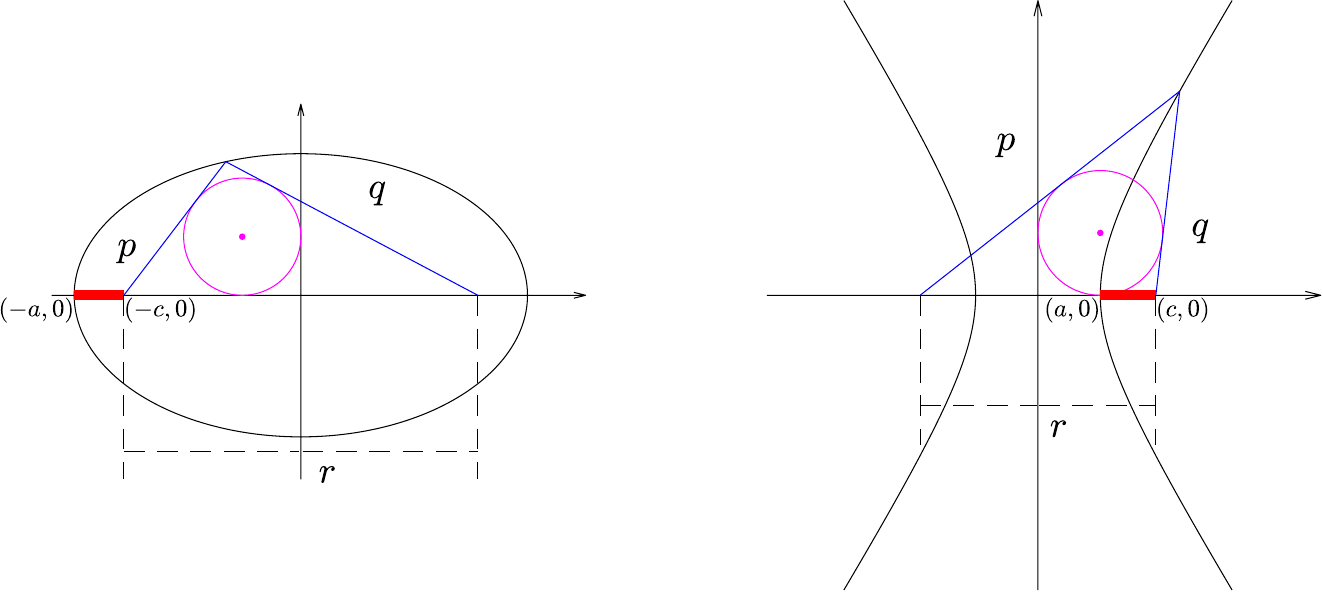}
	
	\par
\hspace{1cm} $r$ the largest side \hspace{3.8cm} $r$ not the largest side
	\caption{Interpretation by conic sections }
	\label{conic}
\end{figure}

In both Figures  \ref{conic}  (1) and  (2),  $(a, 0)$ and  $(c, 0)$ are the coordinates of the right vertex and focal point, respectively.
Consequently, we have \(E_{p, q}^{r} = [|a-c|] + 1\). This means that the multiplicity can be described using the linear eccentricity $|a-c|$.

The following lemma states some properties of \( E_{p, q}^{r} \) and addresses cases involving boundary values.  

\begin{lemma} \label{syy}
	
	\emph{(1)}~~	 $E_{p, q}^{r}=0$ for  $p+q< r$;
	
	\emph{(2)}~~$E_{p, q}^{r}=1$ for $p+q= r$ or $|p-q|= r$;
	
		\emph{(3)}~~When any one of the parameters \( p, q, r \) is zero, the multiplicity \( E_{p,q}^r \) is non-zero if and only if  the remaining two parameters are equal, in which case \( E_{p,q}^r = 1 \). 
		
	\emph{(4)}~~
	The restriction multiplicity $E_{p,q}^r$ does not depend on the degrees $m$ and $n$ of the algebras $P_n \otimes P_m$. More precisely, if it is computed in $P_n \otimes P_m$, its value remains unchanged when $(n, m)$ is replaced by any pair $(n', m')$ such that $n' \ge p,\; m' \ge q,\; n'+m' \ge r$.

	\emph{(5)}~~$	E_{p, q}^{r}=E_{p, q}^{~p+q+|p-q|-r}
$
	
\end{lemma} 
\proof  The conclusions are derived through straightforward calculations. $\qed$

In Lemma \ref{syy}(5), the equation demonstrates symmetry when \(p, q\) are fixed, and \(r\) varies. As depicted in Figure \ref{L4}, we provide a geometric interpretation for this equation. An ellipse is defined by a triangle \(ABC\) with side lengths \(p, q\), and \(r\) (assuming the length of the largest side is \(r\)); then, adjacent to it, a hyperbola is constructed. The right focus of the hyperbola and the right vertex become the left focus \(B\) and the left vertex \(F\) of the ellipse, respectively, with the focal length being \(|p-q|\). Subsequently, point \(D\) is selected on the hyperbola such that \(DB=q\) and \(DE=p\), where point \(E\) is the hyperbola's left focus. As a result, triangle \(DEB\) is formed, in which the length of side \(EB\) equals \(p+q+|p-q|-r\). Equation (4) is valid because, for both the ellipse and the hyperbola, the distance from the focal point to the vertex on one side is represented by \(|FB|\).

\begin{figure}[ht]
	\centering
	\includegraphics[width=0.6\textwidth]{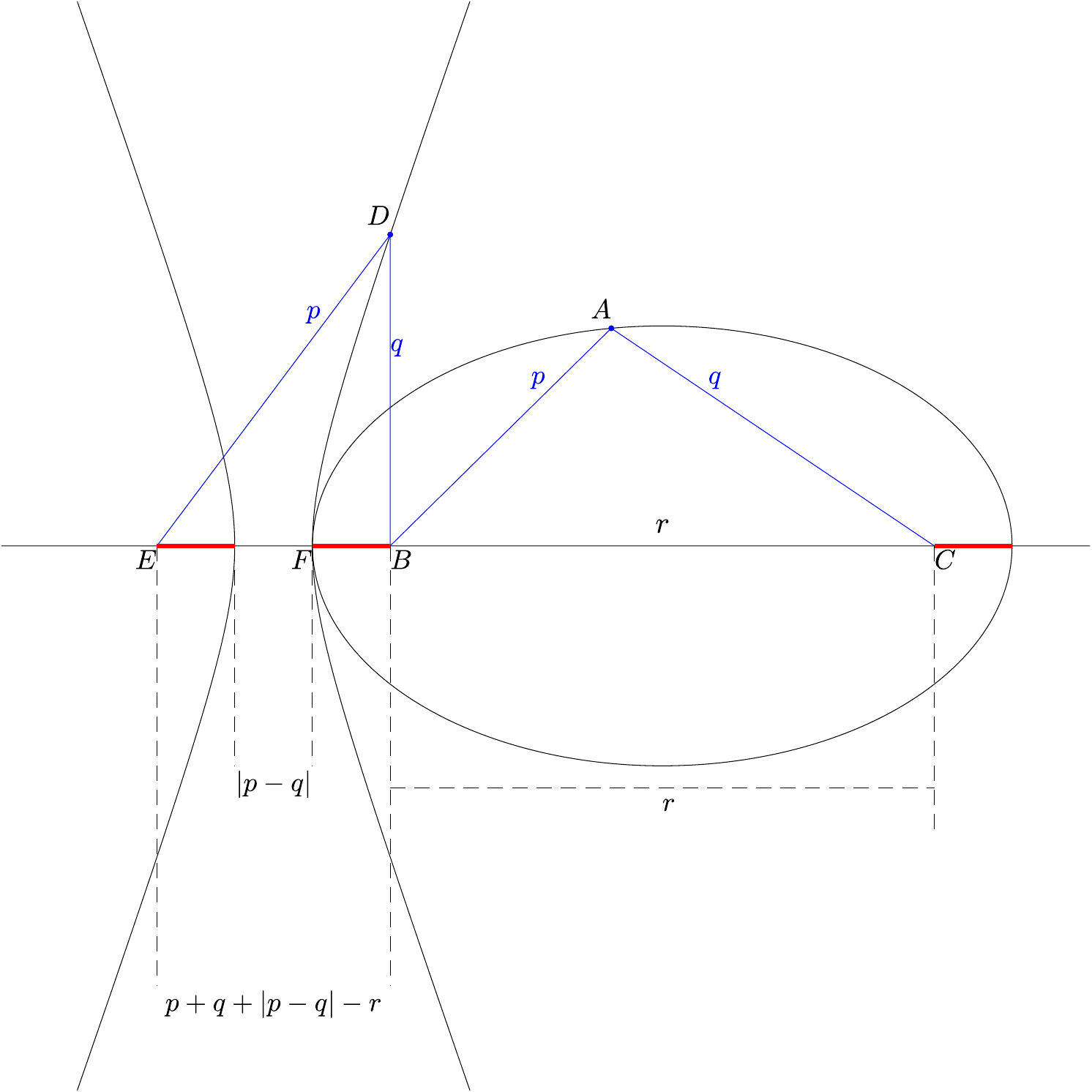}
	\caption{Geometric interpretation of Lemma 3.1 (4) }
	\label{L4}
\end{figure}

\section{Temperley–Lieb algebra  }

The Temperley–Lieb algebra $\mathrm{TL}_n$, which can be realized as a subalgebra of the partition algebra $P_n$, was originally introduced by Temperley and Lieb in the study of problems in statistical mechanics \cite{TL}. It was later rediscovered by Vaughan Jones in his work on subfactors and knot theory \cite{Jones2,Jones85,Jones87}. Kauffman subsequently realized it as a diagram algebra \cite{K1,K2,K3}. The Temperley–Lieb algebra has since appeared in a wide range of mathematical and physical contexts. For example, it plays a role in the construction of link invariants \cite{BW89}, in the representation theory of tensor categories and diagram algebras \cite{BM05,GW93}, and in solvable lattice models in statistical mechanics \cite{Martin90,Martin91,Martin94}. For further details, see the books \cite{GdlHJ89,Bowman} and the survey article \cite{DG,RSA14}.

\begin{figure}[h]
	\begin{equation*}
		\begin{array}{c}
			\begin{tikzpicture}
				\draw [green,thick](-0.4,1.5) --(2.8,1.5);
				\draw [green,thick](-0.4,0) --(2.8,0);
				\draw [green,thick](-0.4,1.5) --(-0.4,0);
				\draw [green,thick](2.8,1.5) --(2.8,0);
				\node at (0,1.8) {$1$}; 
				\node at (0.8,1.8) {$2$};
				\node at (1.6,1.8) {$3$};
				\node at (2.4,1.8) {$4$};
				\node at (0,-0.3) {$1'$}; 
				\node at (0.8,-0.3) {$2'$};
				\node at (1.6,-0.3) {$3'$};
				\node at (2.4,-0.3) {$4'$};
				\node at (0,1.5) {\textbullet}; 
				\node at (0.8,1.5) {\textbullet};
				\node at (1.6,1.5) {\textbullet};
				\node at (2.4,1.5) {\textbullet};
				\node at (0,0) {\textbullet}; 
				\node at (0.8,0) {\textbullet};
				\node at (1.6,0) {\textbullet};
				\node at (2.4,0) {\textbullet};
				\draw [thick](0.8,1.5) .. controls (1,1.1) and (1.4,1.1) ..  (1.6,1.5); 
				\draw [thick](0,1.5) .. controls (0.4,0.7) and (2,0.7) ..  (2.4,1.5); 
				\draw [thick](0,0) .. controls (0.2,0.4) and (0.6,0.4) ..  (0.8,0); 
				\draw [thick](1.6,0) .. controls (1.8,0.4) and (2.2,0.4) ..  (2.4,0); 
				
				\draw [green,thick](3.2,1.5) --(6.4,1.5); 
				\draw [green,thick](3.2,0) --(6.4,0);
				\draw [green,thick](3.2,1.5) --(3.2,0);
				\draw [green,thick](6.4,1.5) --(6.4,0);
				\node at (3.6,1.8) {$1$}; 
				\node at (4.4,1.8) {$2$};
				\node at (5.2,1.8) {$3$};
				\node at (6,1.8) {$4$};
				\node at (3.6,-0.3) {$1'$}; 
				\node at (4.4,-0.3) {$2'$};
				\node at (5.2,-0.3) {$3'$};
				\node at (6,-0.3) {$4'$};
				\node at (3.6,1.5) {\textbullet}; 
				\node at (4.4,1.5) {\textbullet};
				\node at (5.2,1.5) {\textbullet};
				\node at (6,1.5) {\textbullet};
				\node at (3.6,0) {\textbullet}; 
				\node at (4.4,0) {\textbullet};
				\node at (5.2,0) {\textbullet};
				\node at (6,0) {\textbullet};
				\draw [thick](4.4,1.5) .. controls (4.6,1.1) and (5,1.1) ..  (5.2,1.5); 
				\draw [thick](3.6,0) .. controls (3.8,0.4) and (4.2,0.4) ..  (4.4,0); 
				\draw [thick](3.6,1.5) .. controls (3.8,0.8) and (4.8,0.8) ..  (5.2,0); 
				\draw [thick](6,1.5)--(6,0);
			\end{tikzpicture}
		\end{array}
	\end{equation*}
	\vskip -0.3cm
	\caption{Temperley–Lieb diagrams }
	\label{TD}
\end{figure}

In the Temperley–Lieb algebra $\mathrm{TL}_n$, each block of the partition has exactly two elements, and the partition diagram must be planar (i.e. no two blocks cross; see Figure \ref{TD}).
Its half-diagram possesses the following characteristics: there are two types of blocks, one with a single labeled dot and the other containing just two dots, which are unlabeled and typically called a cap. Due to the non-crossing property of the diagram, caps do not intersect, and a labeled dot cannot be `inside' any cap. In particular, for any cap ${i,j}$ with $i<j$ and any labeled single-dot block ${k}$, the forbidden configuration $i<k<j$ can never occur. Figures \ref{2PD} (2) and (3) serve as counterexamples. All half-diagrams with $r$ labeled blocks are denoted by $V_n(r)$. 

\begin{figure}[h]
	\begin{equation*}
		\begin{array}{c}
			\begin{tikzpicture}
				\draw [green,thick](-0.4,0.5) --(3.4,0.5);
				\node at (0,0.8) {$1$}; 
				\node at (0.6,0.8) {$2$};
				\node at (1.2,0.8) {$3$};
				\node at (1.8,0.8) {$4$};
				\node at (2.4,0.8) {$5$};
				\node at (3,0.8) {$6$};
			
				
				\draw [thick](0.6,0.5) .. controls (0.8,0.1) and (2.2,0.1) ..  (2.4,0.5) ; 
				\draw [thick](1.2,0.5) .. controls (1.4,0.3) and (1.6,0.3) ..  (1.8,0.5);
				\node at (1.4,-0.7) { (1)};
				\node at (1.4,-0.2) { \small A half-diagram};
				\node at (0,0.47) {{\huge\textopenbullet}}; 
				\node at (0.6,0.5) {\textbullet};
				\node at (1.2,0.5) {\textbullet};
				\node at (1.8,0.5) {\textbullet};
				\node at (2.4,0.5) {\textbullet};
				\node at (3,0.47) {{\huge\textopenbullet}};
			\end{tikzpicture}
		\end{array}
		\begin{array}{c}
			\begin{tikzpicture}
				\draw [green,thick](-0.4,0.5) --(3.4,0.5);
				\node at (0,0.8) {$1$}; 
				\node at (0.6,0.8) {$2$};
				\node at (1.2,0.8) {$3$};
				\node at (1.8,0.8) {$4$};
				\node at (2.4,0.8) {$5$};
				\node at (3,0.8) {$6$};
	
				\node at (0,0.47) {{\huge\textopenbullet}}; 
				\node at (0.6,0.5) {\textbullet};
				\node at (1.2,0.5) {\textbullet};
				\node at (1.8,0.5) {\textbullet};
				\node at (2.4,0.5) {\textbullet};
				\node at (3,0.47) {{\huge\textopenbullet}};
				\draw [thick](0.6,0.5) .. controls (0.8,0.1) and (1.6,0.1) ..  (1.8,0.5); 
				\draw [thick](1.2,0.5) .. controls (1.4,0.2) and (2.2,0.2) ..  (2.4,0.5);
				\node at (1.4,-0.7) { (2)};
				\node at (1.4,-0.2) { \small \textcolor{red}{Not}  a half-diagram};
			\end{tikzpicture}
		\end{array}
		\begin{array}{c}
			\begin{tikzpicture}
				\draw [green,thick](-0.4,0.5) --(3.4,0.5);
				\node at (0,0.8) {$1$}; 
				\node at (0.6,0.8) {$2$};
				\node at (1.2,0.8) {$3$};
				\node at (1.8,0.8) {$4$};
				\node at (2.4,0.8) {$5$};
				\node at (3,0.8) {$6$};
	
				\node at (0,0.47) {{\huge\textopenbullet}}; 
				\node at (0.6,0.5) {\textbullet};
				\node at (1.2,0.5) {\textbullet};
				\node at (1.8,0.5) {\textbullet};
				\node at (2.4,0.47) {{\huge\textopenbullet}};
				\node at (3,0.5) {\textbullet};
				\draw [thick](0.6,0.5) .. controls (1.4,0) and (2.2,0) ..  (3,0.5); 
				\draw [thick](1.2,0.5) .. controls (1.4,0.2) and (1.6,0.2) ..  (1.8,0.5);
				\node at (1.4,-0.7) { (3)};
				\node at (1.4,-0.2) { \small \textcolor{red}{Not}  a half-diagram};
			\end{tikzpicture}
		\end{array}
	\end{equation*}
	\vskip -0.3cm
	\caption{Half-diagrams}
	\label{2PD}
\end{figure}
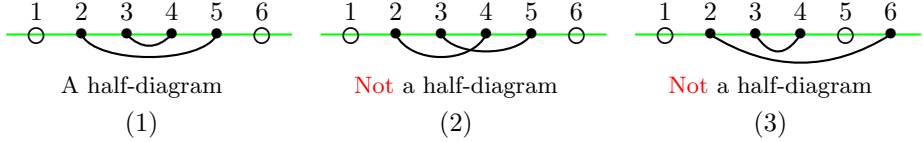

The $(m | n,r)$-walled half-diagram is obtained by adding a wall between dots $m$ and $n'$  with index $r$, taking into account the action of $\mathrm{TL}_m\otimes \mathrm{TL}_n$ on the  diagram. 
We continue to use the previous notation, but it should be noted that for the Temperley--Lieb algebra, the number of through-labeled blocks \(T\) in a walled half-diagram is always equal to \(0\). Furthermore, due to the non-crossing property of Temperley--Lieb diagrams, the part II mentioned in the theorem does not occur. Therefore, the version of Proposition \ref{iso} for the Temperley–Lieb algebra is described as the following \(\mathrm{TL}_m \otimes \mathrm{TL}_n\)-module isomorphism. In article \cite{WP}, we directly proved this isomorphism through the non-crossing property of Temperley--Lieb diagrams.

\begin{prop} [{\cite[Proposition 2.2]{WP}}] \label{forTL}
	As \(\mathrm{TL}_m \otimes \mathrm{TL}_n\)-modules,
	$$V_{m|n}(U~;~ 0,L,R ) \cong V_{m}(U+L)\boxtimes V_n(U+R) .$$
\end{prop} 
Meanwhile, \(T\) in equation (\(*\)) is zero. This implies that \(E_{p, q}^{r}= 1\) if the incircle of the triangle formed by \(p\), \(q\), and \(r\)   intersects the three sides at integer lengths (i.e., as shown in Figure \ref{Tri}, the lengths of the line segments obtained by the incircle intersecting each side are integers). Otherwise, \(E_{p, q}^{r}= 0\).

However, the incircle necessarily cuts each of the three sides into segments of integer length.
This is because, first,
note that  the index \(r\) counts labeled dots. The combinatorial structure of Temperley--Lieb half-diagrams implies that each cap connects two dots, enforcing the parity constraints:
\[
r \equiv m+n \pmod{2}, \quad p \equiv m \pmod{2}, \quad q \equiv n \pmod{2}.
\]

In addition, we have the following geometric fact:
\begin{quote}
	\textbf{Proposition}: For a triangle with integer side lengths \(p, q, r\), the incircle intersects each side at integer lengths if and only if \(p + q + r\) is even.
\end{quote}

From the parity constraints, we derive:
\[
p + q + r \equiv m + n + (m + n) \equiv 2(m + n) \equiv 0 \pmod{2}
\]
Thus, when \(p, q, r\) satisfy the triangle inequalities, the integer tangency condition automatically holds as a consequence of the parity constraints.

\subsection{Grothendieck group } 
Let $\mathrm{TL}(\delta):=\mathrm{TL}_{0}\subset \mathrm{TL}_{1}\subset \mathrm{TL}_{2} \subset \cdots$~be a tower of the Temperley-Lieb algebras $\mathrm{TL}_{n}$ with $\de\in \mathbb{C}$ as a parameter. As we know, when \(\delta\) is not a root of unity, \(\mathrm{TL}_n\) is a semisimple algebra, and all half-diagram modules \(V_{n}(r)\) are precisely all the  irreducible modules. In this context, the Grothendieck group \(G_0(\mathrm{TL}_n)\) can be described as the free abelian group generated by the isomorphism classes \([V_{n}(r)]\) of half-diagram modules \(V_{n}(r)\).  Every module's class $[M]$ in \(G_0(\mathrm{TL}_n)\) can be expressed uniquely as a sum of the classes  \([V_{n}(r)]\), weighted by their multiplicities in the module.

Furthermore, we can define a multiplication on the Grothendieck group, turning it into a ring.
Let $G_{0}(\mathrm{TL}):=\oplus_{n\geqslant 0}G_{0}(\mathrm{TL}_{n})$.  Define the  multiplication on $G_{0}(\mathrm{TL})$ as follows:

$$i_{m,n} :   G_{0}(\mathrm{TL}_{m})\otimes_{\mathbb{Z}} G_{0}(\mathrm{TL}_{n})\rightarrow G_{0}(\mathrm{TL}_{m+n})$$
$$[M]\cdot [N] \mapsto [( M\boxtimes N)\uparrow_{m,n}^{m+n}],$$  where
$$( M\boxtimes N)\uparrow_{m,n}^{m+n}= \mathrm{TL}_{m+n} \bigotimes_{ \mathrm{TL}_{m} \otimes  \mathrm{TL}_{n}}(M\boxtimes N).$$ 

In article \cite{WP}, we proved that the structure constants of this multiplication are precisely the multiplicities \(E_{p, q}^{r}\). Below, in conjunction with Figure \ref{Tri}, we restate this theorem using planar triangle\footnote{The number of caps is indicated by the index in Theorem 1.1 and Proposition 2.2 in reference \cite{WP}. If the number of the labeled blocks is used as the index instead, it is equivalent to Proposition \ref{forTL} and Theorem \ref{cons99} presented here.
	} .

\begin{thm}[\cite{WP}, Theorem 1.1]\label{cons99}
	Let \(\mathrm{TL}(\delta)\) be a tower of the Temperley-Lieb algebras with \(\delta \in \mathbb{C}\) not a root of unity, and let \(V_n(r)\) denote the half-diagram module of \(\mathrm{TL}_{n}\) indexed by the number \(r\) of labeled dots. Then, for each \(m,n \in \mathbb{N}\), 
	the multiplication in the Grothendieck ring \(G_{0}(\mathrm{TL})\) is given by
	\[
	[V_{m}(p)] \cdot [V_{n}(q)] = \sum_{\substack{r=0 \\ r \equiv m+n \pmod{2}}}^{m+n} E_{p,q}^{r} [V_{m+n}(r)],
	\]  
	where the multiplicity coefficient is defined by
	\[
	E_{p,q}^{r} = 
	\begin{cases} 
		1 & \text{if } |p - q| \leq r \leq p + q  \text{ (i.e., \(p, q, r\) form a triangle, including degenerate cases)}, \\
		0 & \text{otherwise}.
	\end{cases}
	\]
\end{thm}

\section*{Acknowledgements}
The authors wish to express their sincere gratitude to colleagues for their constructive comments during the development of this manuscript. Their valuable suggestions have significantly enhanced the quality of this work.


\end{document}